\documentclass[french]{amsart}

\usepackage[french]{babel}
\usepackage[applemac]{inputenc}
\usepackage{latexsym}

\usepackage{amsmath}
\usepackage{amsfonts}
\usepackage{amssymb}
\usepackage{verbatim}
\usepackage{psfrag}
\usepackage{graphicx}

\theoremstyle{theorem}
\newtheorem{lemm}{Lemme}[subsection]
\newtheorem{prop}[lemm]{Proposition}
\newtheorem{coro}[lemm]{Corollaire}
\newtheorem{theo}{Th\'eor\`eme}
\theoremstyle{remark}
\newtheorem{defi}[lemm]{D\'efinition}
\newtheorem{conv}[lemm]{Convention}

\newtheorem{rema}[lemm]{Remarque}

  \usepackage[all]{xy}
\xyoption{matrix}

\newcommand{\Z}[1]{\mathbb{Z}/#1\mathbb{Z}}

\newcommand{\PGL}{\mathrm{PGL}}
\newcommand{\im}{{\bf i}}

\newcommand{\C}{\mathbb{C}}
\newcommand{\Bir}{\mathrm{Bir}}

\newcommand{\Aut}{\mathrm{Aut}}

\newcommand{\GL}{\mathrm{GL}}
\newcommand{\rkPic}[1]{\mathrm{rk\ Pic}(#1)}
\newcommand{\Jon}{\mathbf{dJo}}

\newcommand{\Pn}{\mathbb{P}^2} 
\newcommand{\Sym}{\mathrm{Sym}}

\newcommand{\Pic}[1]{\mathrm{Pic}(#1)}

\newcommand{\drawat}[3]{\makebox[0pt][l]{\raisebox{#2}{\hspace*{#1}#3}}}

\begin{document}
\title{Sous-groupes alg\'ebriques du groupe de Cremona}
\authors{J\'er\'emy Blanc
\thanks{Soutenu par le fonds national suisse de la recherche scientifique (FNRS)}
\address{
Universit\'e de Grenoble I,\\
Institut Fourier, BP 74,\\ 
38402 Saint-Martin d'H\`eres,\\ France}
}
\maketitle
\begin{center}\begin{tabular}{p{10cm}}
\footnotesize {\sc Abstract.}~We give a complete classification of maximal algebraic subgroups of the Cremona group  $\mathrm{Bir}(\mathbb{P}^2)$ and provide algebraic varieties that parametrize the conjugacy classes.\\ \\
\footnotesize {\sc R\'esum\'e.}~Nous donnons une classification compl\`ete des sous-groupes alg\'ebriques maximaux du groupe de Cremona $\mathrm{Bir}(\mathbb{P}^2)$ et explicitons les vari\'et\'es qui param\`etrent les classes de conjugaison.\end{tabular}\end{center}
\section*{Translated introduction}
If $X$ is an algebraic complex variety, we denote by $\Bir(X)$ its group of birational transformations (birational maps from $X$ to itself). We say that a subgroup $G\subset \Bir(X)$ is \emph{algebraic} if there exists a structure of algebraic group on $G$ such that the action $t:G\times X\dasharrow X$ induced by the inclusion of $G$ into $\Bir(X)$ is a rational map. The reader can verify that $t$ satisfies the classical axioms of rational actions of algebraic groups (in particular that $t$ is defined in $e\times X$); the fact that a finite subgroup of $\Bir(X)$ is algebraic may also be verified.

The group $\Bir(\mathbb{P}^1)$ is a well-known algebraic group, isomorphic to $\PGL(2,\C)$. If $n\geq 2$, the group $\Bir(\mathbb{P}^n)$ is not algebraic; this complicates the understanding of the structure of the group.

The algebraic subgroups of $\Bir(\mathbb{P}^n)$ have interested many mathematicians for a long time. For $n=2$, we cite the complete classification of maximal connected algebraic subgroups of $\Bir(\Pn)$ established by F. Enriques \cite{bib:Enr}, where the author proves that every maximal connected algebraic subgroup of $\Bir(\Pn)$ is the conjugate of the identity component of the  automorphism group of a minimal rational surface. A modern proof may be found in \cite{bib:U3}. For $n=3$, the classification of maximal connected algebraic subgroups is due to  F. Enriques and G. Fano; a modern treatment may be found in  \cite{bib:U1}, \cite{bib:U2}, \cite{bib:U3}, and \cite{bib:U4} (the results are also explained in \cite{bib:Oda}). We also cite the study of maximal connected algebraic subgroups of $\Bir(\mathbb{P}^n)$ that contain a torus of dimension $n$ made by M. Demazure \cite{bib:Demazure}.

For non-connected subgroups, only a few of results are known, even in dimension~$2$. A lot of results exist in the case of finite subgroups (see \cite{bib:SK}, \cite{bib:Wim}, \cite{bib:BaB}, \cite{bib:deF}, \cite{bib:BeB}, \cite{bib:BeaPel}, the survey \cite{bib:IskReview} and its references, and more recently \cite{bib:DoI}, \cite{bib:JBCR} \cite{bib:JBSMF}, \cite{bib:JBLin}), but these results do not show which finite groups are maximal algebraic subgroups. There are also some remaining open question in the classification of finite subgroups (see the section "What is left?" of \cite{bib:DoI}), like the parametrization of conjugacy classes and the precise description of automorphism groups of conic bundles.

This article provides a study, with recent tools, of algebraic subgroups of $\Bir(\Pn)$. We show that any algebraic subgroup of $\Bir(\Pn)$ is contained in a maximal algebraic subgroup and give a complete classification of the maximal algebraic subgroups, and a parametrization of the conjugacy classes by algebraic varieties. Some specific examples are precisely described.

Our approach uses the modern viewpoint of Mori's theory and Sarkisov's program, aiming a generalisation in higher dimension (although we use in fact some tools specific to dimension~$2$).

Let us give an outline of the article. In section~\ref{Sec:Res} we remind some definitions and state the main result  (Theorems~\ref{Theo:Class} and \ref{Thm:Param}), which is the classification of maximal algebraic subgroup of the Cremona group, each group being described as a $G$-Mori fibration. We present these groups in Sections~\ref{Sec:AutdP} and \ref{Sec:GaFibr}. We prove then (Section~\ref{Sec:Contenu}) that any algebraic subgroup is contained in one of the groups of the classification, and then that any group of the classification is a minimal $G$-fibration (Section~\ref{Sec:MinP}), that is furthermore birationally supperrigid (Section~\ref{Sec:Superrigidite}). The proof of the Theorems (Section~\ref{Sec:DemoThms}) follows directly from these three observations. We finish the article (Section~\ref{Sec:Z2Max}) with a precise description of a family of maximal algebraic subgroups (the only one which is not directly explicit in Theorems~\ref{Theo:Class} and \ref{Thm:Param}) and a special example (Section~\ref{Sec:Z2pasAlg}) that illustrates some particularities of the classification.

The author would like to express his sincere gratitude to Michel Brion, that asked him the question, and thanks him especially for interesting discussions on this subject. Thanks also to the referees for their precious remarks and corrections.
\section{Introduction}
Si $X$ est une vari\'et\'e alg\'ebrique complexe, on note $\Bir(X)$ le groupe de ses transformations birationnelles (applications birationnelles de $X$ vers lui-m\^eme). On dit qu'un sous-groupe $G\subset \Bir(X)$ est \emph{alg\'ebrique} s'il dispose d'une structure de groupe alg\'ebrique qui rend rationnelle l'action $t:G\times X\dasharrow X$ induite par l'inclusion de $G$ dans $\Bir(X)$. Le lecteur remarquera que $t$ satisfait naturellement les axiomes classiques d'actions rationnelles de groupes alg\'ebriques (en particulier que $t$ est d\'efinie en $e\times X$); on peut \'egalement v\'erifier qu'un sous-groupe fini de $\Bir(X)$ est alg\'ebrique.

Le groupe $\Bir(\mathbb{P}^1)=\Aut(\mathbb{P}^1)$ est un groupe alg\'ebrique bien connu, isomorphe \`a $\PGL(2,\C)$. Lorsque $n\geq 2$, le groupe $\Bir(\mathbb{P}^n)$ n'est pas alg\'ebrique, ce qui rend sa structure plus difficile d'acc\`es.

L'\'etude des sous-groupes alg\'ebriques de $\Bir(\mathbb{P}^n)$ est un sujet abord\'e par un grand nombre de math\'ematiciens depuis d\'ej\`a bien longtemps. Lorsque $n=2$, citons la classification compl\`ete des sous-groupes alg\'ebriques connexes maximaux de $\Bir(\Pn)$ \'etablie par F. Enriques \cite{bib:Enr}, o\`u l'auteur d\'emontre que tout sous-groupe alg\'ebrique connexe maximal de $\Bir(\Pn)$ est le conjugu\'e de la composante connexe de l'identit\'e du groupe des automorphismes d'une surface rationnelle minimale. Une preuve modern se trouve dans \cite{bib:U3}. Lorsque $n=3$, la classification des sous-groupes alg\'ebriques connexes maximaux est due \`a F. Enriques et G. Fano; une preuve moderne se trouve dans \cite{bib:U1}, \cite{bib:U2}, \cite{bib:U3}, et \cite{bib:U4} (les r\'esultats sont aussi expliqu\'es dans \cite{bib:Oda}). Rappelons \'egalement l'\'etude  des sous-groupes alg\'ebriques connexes maximaux de  $\Bir(\mathbb{P}^n)$ qui contiennent un tore de dimension $n$ effectu\'ee par M. Demazure  \cite{bib:Demazure}.

Dans le cas non connexe, peu de r\'esultats sont connus, m\^eme en dimension~$2$. On dispose de nombreuses \'etudes dans le cas de groupes finis (voir \cite{bib:SK}, \cite{bib:Wim}, \cite{bib:BaB}, \cite{bib:deF}, \cite{bib:BeB}, \cite{bib:BeaPel}, le r\'esum\'e \cite{bib:IskReview} et ses r\'ef\'erences, ainsi que plus r\'ecemment \cite{bib:DoI}, \cite{bib:JBCR} \cite{bib:JBSMF}, \cite{bib:JBLin}), mais ces travaux ne permettent pas de d\'eterminer quels groupes finis repr\'esentent des groupes alg\'ebriques maximaux. Il reste aussi quelques points non \'etablis dans la classification des groupes finis (voir la section "What is left?" de \cite{bib:DoI}), comme la param\'etrisation des classes de conjugaison et la description pr\'ecise des groupes d'automorphismes de fibr\'es en coniques. 

Cet article propose une \'etude, \`a l'aide d'outils r\'ecents, des sous-groupes alg\'e\-briques de $\Bir(\Pn)$. Nous prouvons que tout sous-groupe alg\'e\-brique de $\Bir(\Pn)$ est contenu dans un sous-groupe alg\'ebrique maximal et donnons une classification compl\`ete et des sous-groupes alg\'ebriques maximaux, ainsi qu'une param\'etrisation de leurs classes de conjugaison par des vari\'et\'es alg\'ebriques. Quelques exemples caract\'eristiques sont pr\'ecis\'ement d\'ecrits. 

Notre approche utilise le point de vue moderne de la th\'eorie de Mori et du programme de Sarkisov, en esp\'erant une approche similaire en dimension sup\'erieure (tout en utilisant en r\'ealit\'e quelques subtilit\'es propres \`a la dimension~$2$). 

Cet article se pr\'esente de la mani\`ere suivante. A la section~\ref{Sec:Res}, nous rappelons quelques d\'efinitions et \'enonçons le r\'esultat principal (th\'eor\`emes~\ref{Theo:Class} et \ref{Thm:Param}), \`a savoir la classification des sous-groupes alg\'ebriques maximaux du groupe de Cremona, chaque groupe \'etant d\'ecrit en termes de $G$-fibration de Mori. Nous pr\'esentons ces groupes dans les sections~\ref{Sec:AutdP} et \ref{Sec:GaFibr}. Nous d\'emontrons ensuite (section~\ref{Sec:Contenu}) que tout groupe alg\'ebrique est contenu dans un des groupes de la classification, puis que chaque groupe de cette classification donne une $G$-fibration minimale (section~\ref{Sec:MinP}), qui est de plus birationnellement superrigide (section~\ref{Sec:Superrigidite}). La d\'emonstration des th\'eor\`emes (section~\ref{Sec:DemoThms}) d\'ecoule alors de ces trois observations. Nous terminons cet article (section~\ref{Sec:Z2Max}) avec une description pr\'ecise d'une famille de sous-groupes alg\'ebriques maximaux (la seule qui n'est pas directement explicit\'e par les th\'eor\`emes~\ref{Theo:Class} et \ref{Thm:Param}) et un exemple sp\'ecial (section~\ref{Sec:Z2pasAlg}), qui illustre certaines particularit\'es de la classification.

L'auteur tient \`a remercier chaleureusement Michel Brion de lui avoir pos\'e cette question et surtout de lui avoir fait profiter de discussions tr\`es int\'eressantes sur ce sujet. Merci aussi aux rapporteurs pour leurs pr\'ecieuses remarques et corrections.

\section{R\'esultats}\label{Sec:Res}
\subsection{Actions bir\'eguli\`eres sur des vari\'et\'es}
Rappelons que si $G\subset \Bir(\mathbb{P}^n)$ est un sous-groupe alg\'ebrique, alors il existe une application birationnelle  de $\mathbb{P}^n$ vers une vari\'et\'e rationnelle $X$ qui conjugue $G$ \`a un sous-groupe de $\Aut(X)$ (d\'emontr\'e par \cite{bib:Weil} pour le cas connexe,  \'etendu par \cite{bib:Rosenlicht} au cas non connexe, voir aussi \cite{bib:PopovVinberg}).
En passant \`a une compl\'etion \'equivariante (\`a l'aide de \cite{bib:Sum1}), on peut supposer que $X$ est projective; on choisira ensuite $X$ lisse \`a l'aide d'une r\'esolution \'equivariante des singularit\'es. Le groupe $G$ agit alors sur $X$ mais \'egalement sur $\Pic{X}$; on note $\Pic{X}^G$ la partie du groupe de Picard qui est fix\'ee par $G$. En utilisant un programme de Mori $G$-\'equivariant, on veut r\'eduire le rang de $\Pic{X}^G$ et obtenir une $G$-fibration de Mori (voir \cite[Example 2.18]{bib:KoM}), au sens de la d\'efinition suivante:
\begin{defi}\label{defi:GfibreMori}
Soit $X$ une vari\'et\'e projective lisse et $G$ un sous-groupe de $\Aut(X)$. Une application birationnelle $\varphi:X\dasharrow Y$ est \emph{$G$-\'equivariante} si $\varphi G \varphi^{-1}\subset \Aut(Y)$.

Une \emph{$G$-fibration de Mori} est une fibration $\pi:X\rightarrow Y$, o\`u $X$ et $Y$ sont des vari\'et\'es projectives $\mathbb{Q}$-factorielles \`a singularit\'es terminales, $Y$ est normale, $G\subset \Aut(X)$ pr\'eserve  la fibration,
$\dim Y<\dim X$, $\rkPic{X}^G-\rkPic{Y}^G=1$ et $-K_X\cdot C\geq 0$ pour toute courbe $C\subset X$, invariante par $G$ et contract\'ee par $\pi$.

Une $G$-fibration de Mori $\pi:X\rightarrow Y$ est \emph{birationnellement supperrigide} si toute application birationnelle $G$-\'equivariante $\varphi:X\dasharrow X'$ vers une $G$-fibration de Mori $\pi':X'\rightarrow Y'$ est un isomorphisme, induisant de plus un isomorphisme $Y\rightarrow Y'$.
\end{defi}

En dimension $2$ (le cas que nous traiterons ici), la d\'efinition implique que $X$ est une surface projective lisse et $Y$ est soit un point, soit isomorphe \`a $\mathbb{P}^1$. On verra de plus (Proposition~\ref{Prp:Isk}) que dans le premier cas $X$ est de del Pezzo et dans le deuxi\`eme cas, le morphisme $X\rightarrow Y$ est une fibration en coniques.

\subsection{Le programme minimal en dimension $2$}\label{SubSec:MMpDim2}
En dimension $2$, le programme minimal $G$-\'equivariant a \'et\'e \'etabli par Yu. Manin \cite{bib:Man}  (cas ab\'elien) et ensuite par V. Iskovskikh \cite{bib:IskMinimal} (cas g\'en\'eral). Les preuves se font \`a l'aide de couples $(G,S)$ o\`u $S$ est une surface projective lisse et $G$ est un groupe agissant fid\`element sur la surface; ces couples sont \'egalement appel\'es \'egalement $G$-surfaces. Deux cas sont trait\'es de la m\^eme mani\`ere (voir l'introduction de \cite{bib:Man}), \`a savoir lorsque $G$ est un sous-groupe fini de $\Aut(S)$ ou lorsque $G$ est un groupe de Galois. Dans ce dernier cas, le groupe est infini, mais l'action sur le groupe de Picard \'etant finie, les preuves sont similaires. On peut de la m\^eme mani\`ere \'etendre ces r\'esultats au cas o\`u $G$ est un sous-groupe alg\'ebrique de $\Aut(S)$, l'action sur le groupe de Picard \'etant \'egalement finie.

Nous formaliserons ceci avec la proposition~\ref{Prp:Isk} ci-dessous, \'enonc\'ee apr\`es quel\-ques conventions et d\'efinitions que nous utiliserons tout au long de l'article. 

\begin{conv}
On note $\Bir(\Pn)$ le groupe des transformations birationnelles de $\Pn=\Pn(\C)$, appel\'e \'egalement \emph{groupe de Cremona}. 
Dans la suite, toutes les surfaces seront suppos\'ees \emph{projectives, lisses et rationnelles}. Lorsque nous dirons qu'une transformation birationnelle \emph{fixe une courbe}, c'est que sa restriction \`a la courbe est l'identit\'e; si elle envoie juste la courbe birationnellement sur elle-m\^eme, nous dirons qu'elle \emph{laisse la courbe invariante} ou qu'\emph{elle la pr\'eserve}. 
\end{conv}

\begin{defi}\label{Def:Aut(S,pi)}
Si $\pi:X\rightarrow Y$ est un morphisme, on note $\Aut(X,\pi)\subset \Aut(X)$ le groupe des automorphismes qui pr\'eservent la fibration, i.e. qui agissent sur l'ensemble des fibres de $\pi$, et $\Aut(X/Y)\subset \Aut(X,\pi)$ le noyau de cette action.
\end{defi}

\begin{defi}\label{Def:PairesGequi}
Un couple $(G,S)$ est la donn\'ee d'une surface $S$ et d'un groupe $G$, sous-groupe alg\'ebrique de $\Aut(S)$.

On dit que le couple $(G,S)$ est \emph{minimal} si tout morphisme birationnel $G$-\'equivariant $\varphi:S\rightarrow S'$, o\`u $S'$ est une surface, est un isomorphisme.
\end{defi}

\begin{defi}
On dit qu'un morphisme $\pi:S\rightarrow \mathbb{P}^1$ est un \emph{fibr\'e en coniques} si les fibre g\'en\'erales de $\pi$ sont isomorphes \`a $\mathbb{P}^1$, et s'il existe un nombre fini de fibres singuli\`eres, chacune \'etant l'union transverse de deux courbes isomorphes \`a $\mathbb{P}^1$.
\end{defi}

\begin{prop}\label{Prp:Isk}
Soit $(G,S)$ un couple minimal.
Alors, il existe un morphisme $\pi:S\rightarrow Y$ qui soit une $G$-fibration de Mori (au sens de la d\'efinition~\ref{defi:GfibreMori}).

De plus un morphisme 
$G$-\'equivariant $\pi:S\rightarrow Y$ est une  $G$-fibration de Mori, si et seulement si l'un des deux cas suivants se pr\'esente:
\begin{enumerate}
\item
$Y$ est un point, $\Pic{S}^G\otimes\mathbb{Q}=\mathbb{Q}K_S$ et $S$ est une surface de del Pezzo.
\item
$Y\cong \mathbb{P}^1$, $\pi$ est une fibration en coniques  et $\Pic{S}^G\otimes\mathbb{Q}=\mathbb{Q}K_S\oplus \mathbb{Q} f$, o\`u $f$ est la classe de la fibre g\'en\'erale de $\pi$. \qed
\end{enumerate}
\end{prop}
Comme cela fut observ\'e pr\'ec\'edemment, cette proposition suit de \cite{bib:Man} et \cite{bib:IskMinimal}. Depuis, avec les travaux de Mori \cite{bib:Mor} et de ses successeurs sur le programme minimal, on peut trouver plusieurs d\'emonstrations ou \'enonc\'es plus r\'ecents, comme par exemple \cite[Example 2.18]{bib:KoM}.

\bigskip

Rappelons que si $X$ est une vari\'et\'e projective, le noyau $K$ de l'application $\Aut(X)\rightarrow \Aut(\Pic{X})$ est un sous-groupe alg\'ebrique lin\'eaire (se voit \`a l'aide d'un plongement de $X$ dans un espace projectif). Un groupe $G$ tel que $K\subset G \subset \Aut(X)$ est donc un sous-groupe alg\'ebrique si et seulement si son action sur $\Pic{X}$ est finie. 
Ceci nous permet d'\'enoncer la proposition suivante:

\begin{prop}\label{Prp:AutFibMestAlgebrique}
Si $S$ est une surface de del Pezzo, le groupe $\Aut(S)$ est un sous-groupe alg\'ebrique du groupe de Cremona.

Si $\pi:S\rightarrow \mathbb{P}^1$ est un fibr\'e en coniques, alors le groupe $\Aut(S,\pi)$ est un sous-groupe alg\'ebrique du groupe de Cremona.
\end{prop}
\begin{proof}
Le premier cas se voit \`a l'aide du plongement induit par un multiple de l'anti-canonique, qui montre que le groupe $\Aut(S)$ est un sous-groupe alg\'ebrique lin\'eaire; la finitude de l'action sur $\Pic{S}$ est \'egalement prouv\'ee dans \cite{bib:DolWeyl}. Dans le deuxi\`eme cas, la finitude de cette action suit du fait que $\Pic{S}$ est engendr\'e par les composantes de chaque fibre singuli\`ere (qui sont en nombre fini) et par $K_S$.
\qed\end{proof}

Afin de classifier les sous-groupes alg\'ebriques maximaux du groupe de Cremona, nous devons donc d\'eterminer quelles surfaces de del Pezzo et quels fibr\'es en coniques donnent lieu \`a des groupes alg\'ebriques maximaux. Nous verrons que ces groupes doivent agir minimalement sur la surface.

Parmi les fibr\'es en coniques, beaucoup ont des groupes d'automorphismes qui agissent minimalement sur la surface, mais qui ne repr\'esentent pas des groupes alg\'ebriques maximaux. 
Les fibr\'es en coniques repr\'esentant des groupes alg\'ebriques maximaux se divisent en fait en trois familles; nous \'enum\'erons celles-ci maintenant, avant de donner la classification  des sous-groupes alg\'ebriques maximaux du groupe de Cremona.

\subsection{Surfaces de Hirzebruch}Si $n\geq 1$, le groupe $\Aut(\mathbb{F}_n)$ pr\'eserve une unique fibration en coniques (qui est en fait une fibration en droites) sur la surface $\mathbb{F}_n$.

\subsection{Fibr\'es exceptionnels}\label{Subsec:DefFibExc}
On dit qu'un fibr\'e en coniques $\pi:S\rightarrow \mathbb{P}^1$ est \emph{exceptionnel} s'il est singulier au-dessus de $2n$ points ($n\geq 1$) et qu'il existe deux sections d'auto-intersection~$-n$ (cette notion a d\'ej\`a \'et\'e introduite dans \cite{bib:DoI}). Nous verrons (corollaire~\ref{Coro:IsoFibCExcp}) que la classe d'isomorphisme de $(S,\pi)$ est uniquement determin\'ee par la classe des points \`a fibre singuli\`ere, \`a action de $\Aut(\mathbb{P}^1)$ pr\`es. 

\subsection{${(\Z{2})^2}$-fibr\'es en coniques}\label{SubSec:Z22fibr}
On dit qu'un fibr\'e en coniques $\pi:S\rightarrow \mathbb{P}^1$ est un \emph{$(\Z{2})^2$-fibr\'e en coniques} si le groupe $\Aut(S/\mathbb{P}^1)$ est isomorphe \`a $(\Z{2})^2$ et si chacune des trois involutions de ce groupe fixe (point par point) une courbe irr\'eductible, rev\^etement double de $\mathbb{P}^1$ via $\pi$, ramifi\'ee au-dessus d'un nombre positif (pair) de points. 

\begin{defi}
Un \emph{triplet de ramification} est un triplet $\{A_1,A_2,A_3\}$, tel que $A_i$ est un ensemble de $2a_i\geq 2$ points de $\mathbb{P}^1$ pour $i=1,2,3$, et tel que tout point de $\mathbb{P}^1$ appartient \`a $0$ ou $2$ ensembles du triplet.
\end{defi}
Nous verrons (proposition~\ref{Prp:Z2FibRam}) que l'application qui associe \`a un $(\Z{2})^2$-fibr\'e en coniques le triplet form\'e des trois ensembles de points de ramification des trois courbes fix\'ees par les involutions de $\Aut(S/\mathbb{P}^1)$ est une bijection entre les classes d'isomorphisme de $(\Z{2})^2$-fibr\'es en coniques et les triplets de ramification, \`a action de $\Aut(\mathbb{P}^1)$ pr\`es.
Nous d\'emontrerons aussi que $\Aut(S,\pi)$ est un sous-groupe alg\'ebrique maximal du groupe de Cremona si et seulement si la surface $S$ n'est pas une surface de del Pezzo (ce qui est le cas notamment d\`es que le nombre de points de ramification est assez grand). Mentionnons qu'en g\'en\'eral, $\Aut(S,\pi)\not=\Aut(S)$ et que ce dernier groupe n'est pas toujours un groupe alg\'ebrique  (section~\ref{Sec:Z2pasAlg}).

\subsection{La classification}
Nous pouvons maintenant \'enoncer la classification des sous-groupes alg\'ebriques du groupe de Cremona, qui a \'et\'e annonc\'ee pr\'ec\'edemment et qui sera d\'emontr\'ee \`a la section~\ref{Sec:DemoThms}.
\begin{theo}\label{Theo:Class}
Tout sous-groupe alg\'ebrique du groupe de Cremona  est contenu dans un sous-groupe alg\'ebrique maximal.

Les sous-groupes alg\'ebriques maximaux du groupe de Cremona sont les conjugu\'es des groupes $G=\Aut(S,\pi)$ o\`u $S$ est une surface rationnelle et $\pi:S\rightarrow Y$ un morphisme, tels que $(S,\pi)$ soit dans l'un des cas suivants:
\begin{enumerate}
\item
$Y$ est un point, $G=\Aut(S)$ et $S$ est une des surfaces de del Pezzo suivantes:
\begin{enumerate}
\item
$\Pn$, $\mathbb{P}^1\times\mathbb{P}^1$;
\item
une surface de del Pezzo de degr\'e~$1$, $4$, $5$ ou $6$;
\item
une surface de del Pezzo de degr\'e~$3$ (respectivement $2$) telle que le couple $(\Aut(S),S)$ soit minimal (toujours vrai si le degr\'e est $2$) et telle que les points fixes de l'action de $\Aut(S)$ sur $S$ soient sur les courbes exceptionnelles; 
\end{enumerate}
\item
$Y\cong \mathbb{P}^1$ et $\pi$ est une des fibrations en coniques suivantes:
\begin{enumerate}
\item
la fibration en droites de la surface de Hirzebruch $\mathbb{F}_n$, pour $n\geq 2$;
\item
un fibr\'e en coniques exceptionnel ayant au minimum $4$ fibres singuli\`eres;
\item
un ${(\Z{2})^2}$-fibr\'e en coniques, tel que $S$ ne soit pas une surface de del Pezzo (ayant en particulier au moins $6$ fibres singuli\`eres).
\end{enumerate}
\end{enumerate}
De plus, dans chacun des cas ci-dessus, le couple $(G,S)$ est minimal et la fibration $\pi:S\rightarrow Y$ est une $G$-fibration de Mori birationnellement superrigide.

En particulier, deux tels groupes $G=\Aut(S,\pi)$ et $G'=\Aut(S',\pi')$ sont conjugu\'es si et seulement s'il existe un isomorphisme $S\rightarrow S'$ qui envoie les fibres de $\pi$ sur celles de $\pi'$.
\end{theo}
D\'ecrivons maintenant plus pr\'ecis\'ement la structure de ces groupes alg\'ebriques maximaux et \'enonçons aussi que les classes de conjugaison des groupes sont param\'etr\'ees par des vari\'et\'es alg\'ebriques.

\begin{theo}\label{Thm:Param}
Les sous-groupes alg\'ebriques maximaux du groupe de Cremona  appartiennent -- \`a conjugaison pr\`es --  \`a l'une des onze familles suivantes:
\begin{enumerate}
\item
$\Aut(\Pn)\cong\PGL(3,\C)$
\item
$\Aut(\mathbb{P}^1\times\mathbb{P}^1)\cong(\PGL(2,\C))^2\rtimes \Z{2}$
\item
$\Aut(S_6)\cong (\C^{*})^2\rtimes (\Sym_3 \times \Z{2})$, o\`u $S_6$ est la surface de del Pezzo de degr\'e~$6$.
\item
$\Aut(\mathbb{F}_n)\cong\C^{n+1}\rtimes \GL(2,\C)/\mu_n$, o\`u $\mathbb{F}_n$ est la $n$-i\`eme surface de Hirzebruch et $\mu_n$ est la $n$-torsion du centre de $\GL(2,\C)$, avec $n\geq 2$.
\item
$\Aut(S,\pi)$ o\`u $(S,\pi)$ est un fibr\'e en coniques exceptionnel, ayant des fibres singuli\`eres au-dessus d'un ensemble $\Delta\subset \mathbb{P}^1$ de $2n$ points distincts, $n\geq 2$. La projection de $\Aut(S,\pi)$ sur $\PGL(2,\C)$ donne une suite exacte
\[1\rightarrow  \C^{*}\rtimes \Z{2}\rightarrow \Aut(S,\pi) \rightarrow H_{\Delta} \rightarrow 1,\]
o\`u $H_{\Delta}$ est le sous-groupe fini de $\PGL(2,\C)=\Aut(\mathbb{P}^1)$ constitu\'e des \'el\'ements qui laissent l'ensemble $\Delta$ invariant.
 \item
$\Aut(S_5)\cong \Sym_5$ o\`u $S_5$ est la surface de del Pezzo de degr\'e~$5$.
\item
$\Aut(S)\cong (\Z{2})^4\rtimes H_S$ o\`u $S$ est n'importe quelle surface de del Pezzo de degr\'e~$4$, obtenue en \'eclatant $5$ points de $\Pn$, et $H_S$ est le groupe des automorphismes de $\Pn$ qui laissent l'ensemble des points \'eclat\'es invariant.
\item
$\Aut(S)$ o\`u $S$ est une surface de del Pezzo de degr\'e~$3$ (cubique lisse), de la forme suivante:
\begin{enumerate}
\item
le rev\^etement triple de $\Pn$, ramifi\'e le long d'une cubique lisse $\Gamma$. Si $S$ est la cubique de Fermat alors $\Aut(S)=(\Z{3})^3\rtimes \Sym_4$; sinon on a une suite exacte
\[1\rightarrow \Z{3}\rightarrow \Aut(S)\rightarrow H_ \Gamma\rightarrow 1,\]
o\`u $H_{\Gamma}$ est le groupe des automorphismes de $\Pn$ qui laissent la cubique invariant; $H_{\Gamma}$ contient un sous-groupe isomorphe \`a $(\Z{3})^2$;
\item
la surface cubique de Clebsch\footnote{en choisissant $F=\zeta^3W+X+\zeta Y+\zeta^4Z$, o\`u $\zeta$ est une racine $5$-i\`eme de l'unit\'e et $\sigma=(W,X,Y,Z)\mapsto (X,Y,Z,W)$, le morphisme $(W:X:Y:Z)\mapsto (F:F\sigma:F\sigma^2:F\sigma^3:-\sum_{i=0}^3 F\sigma^i)$ induit un isomorphisme entre la surface cubique d'\'equation $WX^2+XY^2+YZ^2+Z^2Y=0$ dans $\mathbb{P}^3$ avec la surface cubique "diagonale" de Clebsch d'\'equation $\sum X_i=\sum (X_i)^3=0$ dans $\mathbb{P}^4$.}, d'\'equation $WX^2+XY^2+YZ^2+Z^2Y=0$ et de groupe d'automorphismes isomorphe \`a $\Sym_5$;
\item
une surface cubique d'\'equation
$W^3+W(X^2+Y^2+Z^2)+\lambda XYZ=0$, pour un $\lambda \in \C$, 
$9\lambda^3\not=8\lambda$, $8\lambda^3\not=-1$, dont le groupe d'isomorphisme est isomorphe \`a $\Sym_4$.
\end{enumerate}
\item
$\Aut(S)\cong \Z{2}\rtimes H_S$ o\`u $S$ est une surface de del Pezzo de degr\'e~$2$, rev\^etement double d'une quartique lisse $Q_S\subset \Pn$ telle que $H_S=\Aut(Q_S)$ agit sans point fixe sur la quartique priv\'ee de ses points de bitangence.
\item
$\Aut(S)$ o\`u $S$ est n'importe quelle surface de del Pezzo de degr\'e~$1$, rev\^etement double d'un cone quadratique $Q$, ramifi\'e le long d'une courbe $\Gamma_S$ de degr\'e~$6$, intersection compl\`ete de $Q$ avec une surface cubique de $\mathbb{P}^3$. On a une suite exacte (en g\'en\'eral non scind\'ee)
\[1 \rightarrow \Z{2} \rightarrow \Aut(S)\rightarrow H_S \rightarrow 1,\] o\`u $H_S$ est le groupe des automorphismes de $Q$ qui laissent la courbe $\Gamma_S$ invariant.
\item
$\Aut(S,\pi)$ o\`u $(S,\pi)$ est un $(\Z{2})^2$-fibr\'e en coniques, tel que $S$ ne soit pas une surface de del Pezzo.

La projection de $\Aut(S,\pi)$ sur $\PGL(2,\C)$ donne une suite exacte
\[1\rightarrow V\rightarrow \Aut(S,\pi) \rightarrow H_{V} \rightarrow 1,\]
o\`u $V\cong (\Z{2})^2$ contient trois involutions $\sigma_1,\sigma_2,\sigma_3$ fixant chacune une courbe hyperrelliptique, ramifi\'ee au-dessus des points de $A_1,A_2,A_3\subset \mathbb{P}^1$ et $H_V\subset \Aut(\mathbb{P}^1)$ est le sous-groupe fini qui laisse l'ensemble $\{A_1,A_2,A_3\}$ invariant.
 \end{enumerate} 
 Les onze familles sont disjointes et les classes de conjugaison au sein de chaque famille sont param\'etr\'ees respectivement par
 \begin{enumerate}
 \item[(1-3)]
Le point.
 \item[(4)]
Il n'y a qu'une classe de conjugaison pour chaque entier $n\geq 2$.
\item[(5)]
Pour chaque entier $n\geq 2$, les ensembles de $2n$ points de $\mathbb{P}^1$, modulo l'action de $\PGL(2,\C)=\Aut(\mathbb{P}^1)$.
 \item[(6)]
Le point.
 \item[(7)]
 Les classes d'isomorphisme de surfaces de del Pezzo de degr\'e~$4$.
 \item[(8)]
Les classes d'isomorphisme des surfaces cubiques en question, donn\'ees respectivement
\begin{enumerate}
\item
par les classes d'isomorphisme des courbes elliptiques;
\item
pour la surface de Clebsch, il n'y a qu'une classe d'isomorphisme;
\item
par les classes de $\{\lambda \in \C\ |\ 9\lambda^3\not=8\lambda, 8\lambda^3\not=-1\}$ modulo l'\'equivalence $\lambda \sim -\lambda$.
\end{enumerate}
 \item[(9)]
Les classes d'isomorphisme des quartiques lisses de $\Pn$ ayant des groupes d'automorphismes agissant sans point fixe sur la quartique priv\'ee de ses points de bitangence.
 \item[(10)]
 Les classes d'isomorphisme de surfaces de del Pezzo de degr\'e~$1$.
\item[(11)]
Les triplets de ramification $\{A_1,A_2,A_3\} \subset \mathbb{P}^1$ d\'eterminant des $(\Z{2})^2$-fibr\'es en coniques sur des surfaces qui ne sont pas de del Pezzo, modulo l'action de $\Aut(\mathbb{P}^1)$.
 \end{enumerate}
\end{theo}
Nous d\'ecrivons plus pr\'ecis\'ement le cas des surfaces de del Pezzo de degr\'e~$2$, o\`u toutes les surfaces ne donnent pas forc\'ement des groupes alg\'ebriques maximaux, au lemme~\ref{Prp:delPezzodeg2}. 
De m\^eme, les $(\Z{2})^2$-fibr\'es en coniques donnant lieu \`a des groupes alg\'ebriques maximaux si et seulement si la surface n'est pas de del Pezzo, ces cas sont d\'etermin\'es \`a la section~\ref{Sec:Z2Max}. 

\section{Automorphismes de surfaces de del Pezzo}\label{Sec:AutdP}
\subsection{G\'en\'eralit\'es}
Rappelons qu'une surface de del Pezzo est isomorphe \`a $\Pn$, $\mathbb{P}^1\times\mathbb{P}^1$ ou \`a l'\'eclatement d'un ensemble $\Delta$ de $r$ points dans le plan, avec $1\leq r \leq 8$ et tels que la surface ne contiennent pas de courbe d'auto-intersection~$\leq -2$. Ceci est \'equivalent \`a ce qu'il n'existe pas de droite passant par $3$ points de $\Delta$, ni de conique passant par $6$ points de $\Delta$, ni de cubique passant par $8$ points de $\Delta$ et \'etant singuli\`ere en l'un des points. Par la suite, nous dirons simplement que les points sont en position g\'en\'erale. 

Pour toute surface $S$ de del Pezzo, le groupe $\Aut(S)$ est un groupe alg\'ebrique (proposition~\ref{Prp:AutFibMestAlgebrique}).
 Nous d\'emontrerons plus tard que $\Aut(S)$ est un sous-groupe maximal du groupe de Cremona si et seulement si le couple $(\Aut(S),S)$ est minimal et que l'action de $\Aut(S)$ sur $S$ priv\'e de ses courbes exceptionnelles a des orbites qui ont toutes une taille au moins \'egale \`a $(K_S)^2$; ces deux conditions \'etant toujours v\'erifi\'ees pour $(K_S)^2\in\{1,4,5,6,9\}$.
 
Le degr\'e d'une surface de del Pezzo est le carr\'e de son diviseur canonique.  Rappelons qu'il existe une unique classe d'isomorphismes de surfaces de del Pezzo de degr\'e~$d$, pour $d=5,6,7,9$, deux classes d'isomorphisme pour $d=8$ (les surfaces de Hirzebruch $\mathbb{F}_0=\mathbb{P}^1\times\mathbb{P}^1$ et $\mathbb{F}_1$) et une infinit\'e de classes d'isomorphisme pour $d=1,2,3,4$.

Les groupes d'automorphismes de $\Pn$ et $\mathbb{P}^1\times\mathbb{P}^1$ sont tr\`es connus; il s'agit des groupes alg\'ebriques $\PGL(3,\C)$ et $(\PGL(2,\C)\times \PGL(2,\C))\rtimes \Z{2}$. Il est assez naturel que ceux-ci soient maximaux; nous d\'emontrerons ceci plus tard (section ~\ref{Sec:DemoThms}). Dans la suite de cette section, $\Delta$ sera un ensemble de $r$ points en position g\'en\'erale, avec $1\leq r\leq 8$, et $S_{\Delta}$ sera la surface de del Pezzo obtenue en \'eclatant $r$ points en position g\'en\'erale de $\Pn$. Si $r=1,2$, le couple $(\Aut(S_{\Delta}),S_{\Delta})$ n'est pas minimal et le groupe $\Aut(S_{\Delta})$ est conjugu\'e \`a un sous-groupe de respectivement $\Aut(\Pn)$ et $\Aut(\mathbb{P}^1\times\mathbb{P}^1)$. Si $r=3,4$ nous noterons $S_{\Delta}=S_{9-r}$ vu que la surface ne d\'epend pas de l'ensemble des points \'eclat\'es mais uniquement du degr\'e, \'egal \`a $9-r$.

Les groupes d'automorphismes des surfaces de del Pezzo ont \'et\'e d\'ecrits \`a de nombreuses reprises depuis le $\mathrm{XIX}^e$ si\`ecle dans de nombreux article, avec des listes plus ou moins pr\'ecises. Citons en particulier les travaux classiques de \cite{bib:SK}, \cite{bib:Wim} et \cite{bib:Seg}, ainsi que le travail r\'ecent \cite{bib:DoI}, o\`u le lecteur trouvera d'autres r\'ef\'erences sur le sujet.

\subsection{La surface de del Pezzo de degr\'e~$6$}\label{SubSec:dPezzo6}
La surface $S_6$ peut \^etre vue comme \[S_6=\{ \big( (x:y:z) , (a:b:c)\big) \in \Pn\times \Pn\ |\ ax=by=cz\},\] et l'\'eclatement $S_6\rightarrow \Pn$ est donn\'e par la premi\`ere projection. 
Le groupe $\Aut(S_6)$ est isomorphe \`a $(\C^{*})^2\rtimes (\Sym_3 \times \Z{2})$, o\`u les facteurs de ce groupe sont donn\'es respectivement par l'action diagonale sur les coordonn\'ees, les permutations simultan\'ees des coordonn\'ees de chaque facteur et la permutation des deux facteurs.

Il y a $6$ courbes exceptionnelles sur $S_6$, correspondant aux courbes exceptionnelles des $3$ points \'eclat\'es et des $3$ droites passant par deux des points. Ces courbes forment un hexagone et la projection de $\Aut(S_6)$ sur  $\Sym_3\times \Z{2}$ donne l'action sur l'hexagone. Il y a $3$ fibrations en coniques sur la surface, et la projection de $\Aut(S_6)$ sur $\Sym_3$ donne l'action sur l'ensemble des trois fibrations.

\subsection{La surface de del Pezzo de degr\'e~$5$}
La surface $S_5$ contient $10$ courbes exceptionnelles, correspondant aux courbes exceptionnelles des $4$ points \'eclat\'es et aux $6$ droites passant par deux des points.
Il y a $5$ ensembles de $4$ courbes exceptionnelles disjointes. Le groupe $\Aut(S_5)$ agit sur celles-ci et est, via cet action, isomorphe au goupe sym\'etrique $\Sym_5$. 
Ce groupe est en fait engendr\'e par le groupe $\Sym_4$ correspondant aux automorphismes de $\Pn$ qui laissent $\Delta$ invariant et par une involution quadratique ayant comme points-bases $3$ des $4$ points de $\Delta$ et fixant le quatri\`eme.

\subsection{Les surfaces de del Pezzo de degr\'e~$4$}\label{SubSec:dPdeg4}
Supposons que  $r=5$. La surface $S_{\Delta}$ est isomorphe \`a la sous-vari\'et\'e de $\mathbb{P}^4$, donn\'ee par les \'equations
\[\sum_{i=0}^4 x_i^2=\sum_{i=0}^4 \lambda_i x_i^2=0,\]
pour des $\lambda_i$ tous diff\'erents  (voir par exemple \cite[Lemme~6.5]{bib:DoI}).
La $2$-torsion  du sous-groupe diagonal de $\Aut(\mathbb{P}^4)=\PGL(5,\C)$ -- que nous noterons $T\cong (\Z{2})^4$ -- agit sur $S_{\Delta}$ et est engendr\'ee par cinq involutions fixant chacune une courbes elliptique -- trace de l'\'equations $x_i=0$ sur $S_{\Delta}$ pour $i=0,...,4$. 
Il existe $10$ fibrations en coniques sur la surface $S_{\Delta}$, r\'eunies en $5$ couples. Chaque couple correspond au pinceau des droites par un des points de $\Delta$ et au pinceau des coniques passant par les $4$ autres points. Le groupe $\Aut(S_{\Delta})$ agit sur ces $5$ couples  et cette action donne une suite exacte (voir \cite{bib:JBLin}) 
\[1\rightarrow T\rightarrow \Aut(S)\rightarrow H_{\Delta}\rightarrow 1,\]
o\`u $H_{\Delta}$ correspond au groupe des automorphismes de $\mathbb{P}^2$ qui laissent l'ensemble $\Delta$ invariant. Il existe donc une section \`a cette suite exacte et $\Aut(S)=T\rtimes H_{\Delta}$.
Remarquons qu'en g\'en\'eral $H_{\Delta}$ est trivial et que sinon, il peut \^etre isomorphe \`a $\Z{2}$, $\Z{4}$, ou au groupes di\'edraux \`a $6$ ou $10$ \'el\'ements.
\subsection{Les surfaces de del Pezzo de degr\'e~$3$}
Pour $r=6$, l'ensemble des surfaces $S_{\Delta}$ forme la famille des surfaces de del Pezzo qui est probablement la plus connue, puisqu'il s'agit en fait des surfaces cubiques lisses de $\mathbb{P}^3$. Les groupes des automorphismes de telles surfaces ont \'et\'e classifi\'es par Kantor et Wiman, puis corrig\'es respectivement par Segre \cite{bib:Seg} et Hosoh \cite{bib:Ho2}; le lecteur trouvera un historique  dans \cite{bib:DoI}. Une surface cubique g\'en\'erale n'ayant pas d'automorphismes, les couples $(\Aut(S_{\Delta}),S_{\Delta})$ qui repr\'esentent des groupes alg\'ebriques maximaux correspondent \`a une des familles bien particuli\`eres de surfaces cubiques.

\begin{prop}\label{Prp:CubicFibCpasMin}
Soit $S\subset \mathbb{P}^3$ une cubique lisse et soit $\pi:S\rightarrow \mathbb{P}^1$ une fibration en coniques. Alors, le couple $(\Aut(S,\pi),S)$ n'est pas minimal.
En particulier, si $G\subset \Aut(S)$, le couple $(G,S)$ est minimal si et seulement si $\rkPic{S}^G=1$.
\end{prop}
\begin{proof}
Une fibre g\'en\'erale de $\pi$ est une conique planaire de $\mathbb{P}^3$, la droite r\'esiduelle dans le plan associ\'e est invariant par $\Aut(S,\pi)$ et peut \^etre contract\'ee de mani\`ere $\Aut(S,\pi)$-\'equivariante; la paire $(\Aut(S,\pi),S)$ n'est donc pas minimale.
La derni\`ere assertion suit de la proposition~\ref{Prp:Isk}.
\qed\end{proof}

À l'aide de cette proposition, on peut d\'eterminer parmi les groupes d'automorphismes de surfaces cubiques ceux qui donnent des couples minimaux en d\'eterminant leur action sur le groupe de Picard. Ceci se calcule en utilisant la formule de Lefschetz, la repr\'esentation des \'el\'ements comme \'el\'ements du groupe de Weyl, ou directement en trouvant des courbes invariantes. Nous pr\'ef\'erons omettre une telle recherche, longue et fastidieuse, qui peut se retrouver dans \cite{bib:DoI}. Le r\'esultat suivant en d\'ecoule:
\begin{prop}\label{Prp:CubicMin}
Soit $S\subset \mathbb{P}^3$ une surface cubique lisse. Alors le couple $(\Aut(S),S)$ est minimal si et seulement si $S$ est isomorphe \`a une cubique dont l'\'equation est l'une des suivantes
\begin{enumerate}
\item
$W^3+X^3+Y^3+Z^3+\alpha XYZ=0$
\item
$W^2X+X^2Y+Y^2Z+Z^2W=0$
\item
$W^3+W(X^2+Y^2+Z^2)+\beta XYZ=0$
\item
$W^3+X^3+\gamma WX(Y+\delta Z)+Y^3+Z^3=0$
\end{enumerate}
pour $\alpha,\beta,\gamma,\delta\in\C$, avec $9\beta^3\not=8\beta$, $8\beta^3\not=-1$.
Dans le cas $(2)$, $\Aut(S)\cong \Sym_5$ et dans le cas $(3)$, $\Aut(S)\cong \Sym_4$.\qed
\end{prop}
Nous d\'emontrerons (\`a la section~\ref{Sec:DemoThms}) qu'un tel groupe est un sous-groupe alg\'ebrique maximal du groupe de Cremona si et seulement si chacun de ses points fixes sur $S$ appartient  \`a une courbe exceptionnelle. Le lemme suivant d\'ecrit ces cas.
\begin{lemm}\label{Lem:Min3pts}
Soit $S\subset \mathbb{P}^3$ une surface cubique lisse, telle que le couple $(\Aut(S),S)$ soit minimal. 

Si $S$ est isomorphe \`a une surface dont l'\'equation est l'une des trois premi\`eres de la proposition~\ref{Prp:CubicMin}, alors les orbites du groupe $\Aut(S)$ sur $S$ ont au minimum $3$ points distincts.

Sinon, le groupe $\Aut(S)$ fixe un point de $S$ qui n'appartient \`a aucune courbe exceptionnelle.
\end{lemm}
\begin{proof}
D'apr\`es la proposition~\ref{Prp:CubicMin}, l'\'equation de la surface est -- \`a isomorphisme pr\`es -- de l'une des quatre formes cit\'ees dans la proposition~\ref{Prp:CubicMin}. Notons $\omega=e^{2\im\pi/3}$, une racine cubique de l'unit\'e et examinons chacune des $4$ possibilit\'es.

Dans le premier cas, l'automorphisme $W\mapsto \omega W$ a des orbites de taille $3$, sauf sur la courbe elliptique plane $W=X^3+Y^3+Z^3+\lambda XYZ=0$. L'automorphisme $(W:X:Y:Z)\mapsto  (W:X:\omega Y:\omega^2 Z)$ agit sans point fixe sur cette courbe elliptique, ce qui implique le r\'esultat.
Dans le deuxi\`eme cas, l'automorphisme $(W:X:Y:Z)\mapsto (W:\zeta X: \zeta^{-1} Y:\zeta^{3}Z)$, o\`u $\zeta$ est une racine $5$-i\`eme de l'unit\'e, agit sur la surface et fixe $4$ points. Ceux-ci \'etant permut\'es transitivement par l'automorphisme $(W:X:Y:Z)\mapsto (X:Y:Z:W)$, toutes les orbites ont une taille au moins \'egale \`a $4$.
Dans le troisi\`eme cas, observons tout d'abord que $\beta\not=0$ (car la surface est lisse et donc irr\'eductible). Les points de $S$ fix\'es par l'automorphisme $(W:X:Y:Z) \mapsto (W:Y:Z:X)$ sont de la forme $(u:v:v:v)$ pour $u,v \in \C^{*}$. Or, le groupe engendr\'e par $(W:X:Y:Z) \mapsto (W:X:-Y:-Z)$ et $(W:X:Y:Z) \mapsto (W:-X:-Y:Z)$ envoie chacun de ces points sur $3$ autres points diff\'erents.

Dans le dernier cas, le groupe $\Aut(S)$ contient le groupe $H$ isomorphe \`a $\Sym_3$, engendr\'e par $(W:X:Y:Z)\mapsto (\omega W:\omega^2 X:Y:Z)$ et $(W:X:Y:Z)\mapsto (X:W:Y:Z)$. Suivant les valeurs de $\mu$ et $\lambda$, on a trois possibilit\'es (voir \cite{bib:DoI} ou \cite{bib:Seg}): ou bien $\Aut(S)=H$ (qui est le cas g\'en\'eral), ou bien $\mu^3=1$ et $\Aut(S)$ est engendr\'e par $H$ et $(W:X:Y:Z) \mapsto (W:X:\mu Z:\mu^2Y)$, ou bien $\lambda$ et $\mu$ sont tels que $S$ est isomorphe \`a l'une des surfaces pr\'ec\'edentes (notamment lorsque $\lambda=0$). On peut supposer que l'on est dans l'un des deux premiers cas; les points $(0:0:1:-\omega^i)$ pour $i=0,1,2$ sont alors fix\'es par $\Aut(S)$. On v\'erifie par un calcul direct qu'au moins un des trois points n'appartient \`a aucune des $27$ droites de la surface.
\qed\end{proof}

\subsection{Les surfaces de del Pezzo de degr\'e~$2$}\label{SubSec:dP2}
Lorsque $r=7$, le morphisme anti-canonique donne un rev\^etement double $S_{\Delta}\rightarrow \Pn$, ramifi\'e le long d'une quartique lisse (voir \cite{bib:BeauLivre}). R\'eciproquement, toute telle surface est une surface de del Pezzo de degr\'e~$2$. L'involution li\'ee au rev\^etement est appel\'ee \emph{involution de Geiser} et l'action sur la quartique donne lieu a la suite exacte 
\[1\rightarrow \Z{2}\rightarrow \Aut(S_{\Delta}) \rightarrow H_{S_{\Delta}}\rightarrow 1\]
o\`u $H_{S_\Delta}$ est le groupe des automorphismes de $\mathbb{P}^2$ qui laissent la quartique invariante, ou de mani\`ere \'equivalente le groupe des automorphismes de la quartique. On peut de plus observer que cette suite exacte est scind\'ee et donc que $\Aut(S_{\Delta})\cong \Z{2}\times H_{S_{\Delta}}$.

L'involution de Geiser fixant une courbe de genre $3$, elle n'agit bir\'eguli\`erement sur aucune surface de del Pezzo de degr\'e~$\geq 3$. Le couple $(\Aut(S_{\Delta}),S_{\Delta})$ est donc toujours minimale. 
Nous d\'emontrerons (\`a la section~\ref{Sec:DemoThms}) que $\Aut(S_{\Delta})$ est un sous-groupe alg\'ebrique maximal du groupe de Cremona si et seulement si chacun de ses points fixes sur $S_{\Delta}$ appartient  \`a une courbe exceptionnelle. Le lemme suivant d\'ecrit ces cas, reprenant la classification des automorphismes de surfaces de del Pezzo de degr\'e $2$ de \cite{bib:DoI}, inspir\'ee des r\'esultats sur les quartiques lisses.

\begin{lemm}\label{Prp:delPezzodeg2}
Soit $S$ une surface de del Pezzo de degr\'e~$2$, plong\'e dans $\mathbb{P}(2,1,1,1)$ comme $W^2=F(X,Y,Z)$, o\`u $F$ est l'\'equation d'une quartique lisse. Si les points-fixes de l'action  de $\Aut(S)$ sur $S$ sont sur les courbes exceptionnelles, alors $S$ est isomorphe \`a l'une des surfaces  suivantes (les notations pour la structure de $\Aut(S)$ sont celles de \cite{bib:DoI}, qui reprend celles de l'Atlas des groupes finis).\begin{center}
\begin{tabular}{| l|l | l | l | l|l|}
\hline
&\footnotesize{Ordre}&\footnotesize{Structure} &\footnotesize{\'Equation}&\footnotesize{Restrictions}\\ 
&\footnotesize{$|\Aut(S)|$}&\footnotesize{de $\Aut(S)$} &\footnotesize{de la surface $S$}&\footnotesize{param\`etres}\\ \hline
1&336 &\small{$2\times L_2(7)$}&\footnotesize{$W^2=X^3Y+Y^3Z+Z^3X$}&\\ \hline
2&192&\small{$2\times (4^2:S_3)$}&\footnotesize{$W^2=X^4+Y^4+Z^4$}&\\ \hline
3&96&\small{$2\times 4A_4$}&\footnotesize{$W^2=X^4+aX^2Y^2+Y^4+Z^4$}&\footnotesize{$a^2 = -12$}\\ \hline
4&48&\small{$2\times S_4$}&\footnotesize{$W^2=X^4+Y^4+Z^4$}&\footnotesize{$a\ne \frac{-1\pm \sqrt{-7}}{2}$}\\
&&&\footnotesize{$\hphantom{W^2=}a(X^2Y^2+X^2Z^2+Y^2Z^2)$}&\\\hline
5&32&\small{$2\times AS_{16}$}&\footnotesize{$W^2=X^4+aX^2Y^2+Y^4+Z^4$}&\footnotesize{$a^2\ne 0,-12$}\\ 
&&&&\footnotesize{$a^2\ne 4,36$}\\ \hline
6&16&\small{$2\times D_8$}&\footnotesize{$W^2=X^4+Y^4+Z^4+$}&\footnotesize{$a,b\ne 0 $}\\  
&&&\footnotesize{$\hphantom{W^2=}aX^2Y^2 +bZ^2XY$}&\\\hline
7&12&\small{$2\times S_3$}&\footnotesize{$W^2=Z^4+aZ^2XY+$}& \\ 
&&&\footnotesize{$\hphantom{W^2=}Z(X^3+Y^3)+ bX^2Y^2$}&\\\hline
8&$8$&$2^3$&\footnotesize{$W^2=X^4+Y^4+Z^4$}&\footnotesize{$a,b,c\ne 0$}\\
&&&\footnotesize{$\hphantom{W^2=}aX^2Y^2+bX^2Z^2+cY^2Z^2$}&\footnotesize{distincts}\\ \hline
\end{tabular}
\end{center}
\end{lemm}
\begin{proof}
La classification des groupes d'automorphismes des surfaces de del Pezzo de degr\'e $2$, o\`u de mani\`ere \'equivalente celles des quartiques lisses du plan, a \'et\'e \'etablie par Kantor et Wiman, puis revue et corrig\'ee par I.~Dolgachev \cite{bib:DolgTopic} et F.~Bars \cite{bib:Bar};  ceci est r\'esum\'e dans \cite{bib:DoI}. Nous reprenons simplement cette derni\`ere r\'ef\'erence et enlevons les cas $VI$, $VIII$, $XI$, $XII$, $XIII$, qui ont des points-fixes qui n'appartiennent \`a aucune courbe exceptionnelle (rappelons que ces derni\`eres correspondent aux bitangentes de la quartique). On v\'erifie \`a la main que les cas $VII$, $IX$ et $X$ de \cite{bib:DoI} (ici not\'es $6$, $7$ et $8$) n'admettent aucun point fixe.

Le fait que les groupes $1$ \`a $5$ n'admettent aucun point fixe suit de l'observation suivante: si un sous-groupe fini $G\subset\PGL(3,\C)=\Aut(\Pn)$ pr\'eserve une courbe lisse et fixe un point $P$ de la courbe, alors l'action sur l'espace tangent $T_P$ donne une suite exacte $1\rightarrow G'\rightarrow G\rightarrow H\rightarrow 1$ o\`u $G$ et $H$ sont cycliques.
En effet le groupe $H$ agit sur $\mathbb{P}(T_P)\cong \mathbb{P}^1$ avec un point fixe (correspondant \`a la tangente de la courbe) et donc est cyclique, car fini; le groupe $G'$ est cyclique car il s'identifie \`a un sous-groupe de $\mathbb{C}^*$ agissant sur $T_P$ par homoth\'eties.
\qed\end{proof}
\subsection{Les surfaces de del Pezzo de degr\'e~$1$}
Si $r=8$, le morphisme donn\'e par le double de l'anti-canonique donne un morphisme $S_{\Delta}\rightarrow \mathbb{P}^3$, dont l'image est un c\^one quadratique $Q$ et tel que la restriction $S_{\Delta}\rightarrow Q$ est un rev\^etement double, ramifi\'e le long d'une courbe lisse $\Gamma_{\Delta}$, intersection compl\`ete de $Q$ avec une surface de degr\'e~$3$ (voir \cite{bib:BeauLivre}). On dispose \`a nouveau d'une suite exacte naturelle
\[1\rightarrow \Z{2}\rightarrow \Aut(S_{\Delta}) \rightarrow H_{S_{\Delta}}\rightarrow 1\]
o\`u le noyau est engendr\'e par la classique \emph{involution de Bertini} et $H_{S_{\Delta}}$ est le groupe des automorphismes de $Q$ qui pr\'eservent la courbe $\Gamma_{\Delta}$ ou de mani\`ere \'equivalente le groupe des automorphismes de $\Gamma_{\Delta}$. Remarquons qu'en g\'en\'eral la suite n'est pas scind\'ee (on peut notamment trouver des racines carr\'ees de l'involution de Bertini -- voir \cite[Table I, classe [1.B2.2]{bib:JBCR}] et \cite[Table 8, Type XIX]{bib:DoI}).

Les groupes d'automorphismes de surfaces de del Pezzo de degr\'e~$1$ ont \'et\'e classifi\'es par Kantor et Wiman; une approche moderne se trouve dans \cite{bib:DoI}.

\section{Groupes alg\'ebriques pr\'eservant une fibration rationnelle}\label{Sec:GaFibr}
\subsection{G\'en\'eralit\'es}
Une \emph{pseudo-fibration rationnelle} est une application rationnelle $\pi:S\dasharrow \mathbb{P}^1$ dont la fibre g\'en\'erale est une courbe rationnelle. Si de plus $\pi$ est un morphisme, alors nous dirons que c'est une \emph{fibration rationnelle}. Si chaque fibre singuli\`ere est isomorphe \`a une r\'eunion de deux courbes isomorphes \`a $\mathbb{P}^1$ se coupant transversalement en un point, alors $\pi$ est une \emph{fibration en coniques} et si de plus il n'existe pas de fibre singuli\`ere, alors $\pi$ est une \emph{fibration en droites}.

Un th\'eor\`eme de Noether-Enriques \cite[Th\'eor\`eme III.4]{bib:BeauLivre} implique que pour toute pseudo-fibration rationnelle $\pi:S\dasharrow \mathbb{P}^1$ il existe une application birationnelle $S\dasharrow \Pn$ qui envoie les  fibres g\'en\'erales de $\pi$ sur les droites passant par un point; o\`u de mani\`ere \'equivalente qu'il existe une application birationnelle $S\dasharrow \mathbb{P}^1\times \mathbb{P}^1$ qui envoie  les  fibres g\'en\'erales de $\pi$ sur celles d'une des deux fibrations en droites de $\mathbb{P}^1\times \mathbb{P}^1$.

\begin{defi}
Un fibr\'e en coniques $\pi:S\rightarrow \mathbb{P}^1$ se note \'egalement $(S,\pi)$.

Soient $(S,\pi)$ et $(S',\pi')$ deux fibr\'es en coniques. Une application birationnelle $\varphi:S\dasharrow S'$ est une \emph{application birationnelle de fibr\'es en coniques} si et elle envoie une fibre g\'en\'erale de $\pi$ sur une fibre g\'en\'erale de $\pi'$.
\end{defi}
Rappelons que l'ensemble des automorphismes du fibr\'e en coniques $(S,\pi)$ est not\'e $\Aut(S,\pi)$ et que ce dernier est un sous-groupe alg\'ebrique du groupe de Cremona (proposition~\ref{Prp:AutFibMestAlgebrique}).

\begin{defi}\label{Def:Jonq}
On note $\Jon$ le \emph{groupe de de Jonqui\`eres}, constitu\'e des applications de la forme
\[(x,y)\dasharrow \left(\frac{ax+b}{cx+d},\frac{\alpha(x)y+\beta(x)}{\gamma(x)+\delta(x)}\right),\]
pour $\left(\begin{array}{cc}a& b\\ c &d\end{array}\right)\in\PGL(2,\C)$ et $\left(\begin{array}{cc}\alpha& \beta\\ \gamma &\delta \end{array}\right)\in\PGL(2,\C(x))$. 
Ce groupe est naturellement isomorphe \`a
$\PGL(2,\C(x))\rtimes \PGL(2,\C),$
o\`u $\PGL(2,\C)=\Aut(\mathbb{P}^1)$ agit sur $\C(x)=\C(\mathbb{P}^1)$ via l'action de $\Aut(\mathbb{P}^1)$ sur $\mathbb{P}^1$.
\end{defi}

En appliquant \`a nouveau le th\'eor\`eme de Noether-Enriques, nous trouvons que pour toute fibration en coniques $(S,\pi)$, le groupe alg\'ebrique $\Aut(S,\pi)$ est conjugu\'e \`a un sous-groupe du \emph{groupe de de Jonqui\`eres} $\Jon$, introduit \`a la d\'efinition~\ref{Def:Jonq}.

R\'eciproquement, si $G$ est un sous-groupe alg\'ebrique du groupe de de Jonqui\`eres, alors $G$ est conjugu\'e \`a un sous-groupe alg\'ebrique de $\Aut(S,\pi)$, pour un certain fibr\'e en coniques $(S,\pi)$.

\begin{lemm}\label{Lem:ZKZf}
Soit $(S,\pi)$ un fibr\'e en coniques avec au moins une fibre exceptionnelle. Si $\Pic{S}^{\Aut(S,\pi)}$ est de rang $2$, il est \'egal \`a $\mathbb{Z}K_S\oplus \mathbb{Z}f$, o\`u $f$ est le diviseur d'une fibre de $\pi$.
\end{lemm}
\begin{proof}Soit $C\in \Pic{S}^{\Aut(S,\pi)}$; le groupe $\Pic{S}^{\Aut(S,\pi)}$ \'etant de rang $2$, \'ecrivons $C=aK_S+bf$, pour $a,b\in \mathbb{Q}$. L'intersection de $C$ avec une composante d'une fibre singuli\`ere valant $-a$, nous en d\'eduisons que $a\in \mathbb{Z}$.  Ceci montre que $aK_S\in \Pic{S}$ et donc que $bf\in \Pic{S}$, d'o\`u la relation $b\in\mathbb{Z}$.
\qed\end{proof}

\bigskip

Nous pr\'esentons maintenant les sous-groupes alg\'ebriques du groupe de de Jonqui\`eres annonc\'es dans l'introduction. Nous d\'emontrerons plus tard (proposition~\ref{Prp:ContenuFibreConiques}) que les sous-groupes alg\'ebriques maximaux de $\Bir(\Pn)$ qui pr\'eservent une fibration rationnelle sont les conjugu\'es de ces groupes.
\subsection{Automorphismes des surfaces de Hirzebruch}\label{SubSec:AutHirz}
Si $\pi:S\rightarrow \mathbb{P}^1$ est une fibration en droites, alors $S$ est isomorphe \`a une surface de Hirzebruch.
On note $\mathbb{F}_n$ la surface de Hirzebruch $\mathbb{P}_{\mathbb{P}^1}(\mathcal{O}_{\mathbb{P}^1} \oplus \mathcal{O}_{\mathbb{P}^1}(n))$, pour $n\geq 0$. 

La surface $\mathbb{F}_0$ est isomorphe \`a $\mathbb{P}^1\times\mathbb{P}^1$ et il existe deux fibrations en droites sur cette surface. Le groupe des automorphismes de $\mathbb{F}_0$ qui pr\'eservent une des deux fibrations (et donc l'autre) est isomorphe \`a $(\PGL(2,\C))^2$. Ce groupe n'est pas maximal, vu qu'il est contenu dans le groupe 
$\Aut(\mathbb{F}_0)\cong (\PGL(2,\C))^2\rtimes \Z{2}$, o\`u $\Z{2}$ correspond \`a la permutation des facteurs.

Si $n\geq 1$, alors il existe une unique fibration en coniques sur $\mathbb{F}_n$ (qui est en fait une fibration en droites), et le groupe  $\Aut(\mathbb{F}_n)$ pr\'eserve donc celle-ci. Ce groupe est isomorphe \`a  $\C^{n+1}\rtimes \GL(2,\C)/\mu_n$, o\`u $\C^{n+1}$ s'identifie \`a l'ensemble des polyn\^omes homog\`enes de degr\'e~$n$ en deux variables et o\`u $\GL(2,\C)$ agit naturellement sur les variables, avec comme noyau le sous-groupe $\mu_n$ des matrices \'egales \`a la matrice identit\'e multipli\'ee par une racine $n$-i\`eme de l'unit\'e. En coordonn\'ees affines, l'isomorphisme associe \`a \[\left(a_0 Y^n+a_1 XY^{n-1}+...+a_n X^n,\small{\left(\begin{array}{cc} a& b\\ c &d\end{array}\right)}\right)\in \C^{n+1}\rtimes \GL(2,\C)\] l'automorphisme
\[(x,y)\dasharrow \left(\frac{ax+b}{cx+d},\frac{y+a_0+a_1 x+...+a_n x^n}{(cx+d)^n}\right)\]
de la surface $\mathbb{F}_n$, vu comme \'el\'ement du groupe de de Jonqui\`eres.

Le groupe $\Aut(\mathbb{F}_n)$ est un sous-groupe alg\'ebrique connexe du groupe de Cremona. Si $n=1$, alors ce groupe n'est pas maximal vu qu'il est conjugu\'e \`a un sous-groupe de $\Aut(\Pn)$, via la contraction de la section exceptionnelle. Nous verrons (\`a la section~\ref{Sec:DemoThms}) que pour $n\geq 2$, le groupe $\Aut(\mathbb{F}_n)$ est maximal.

\subsection{Automorphismes de fibr\'es exceptionnels}\label{SubSec:AutFibExc}
Nous avons d\'efini en~\ref{Subsec:DefFibExc} ce qu'\'etait un fibr\'e exceptionnel. Nous rappelons bri\`evement quelques descriptions simples des fibr\'es exceptionnels (lemme~\ref{Lem:EquiFibExc}) et de leurs automorphismes (lemme~\ref{Lem:AutExcGDelta}); la plupart de ces r\'esultats se retrouvent \'egalement \`a partir de la partie 5.3 de \cite{bib:DoI}. 
\begin{lemm}\label{Lem:EquiFibExc}
Soit $\pi:S\rightarrow \mathbb{P}^1$ un  fibr\'e en coniques avec $2n$ fibres singuli\`eres. Alors, les conditions suivantes sont \'equivalentes:
\begin{enumerate}
\item
le fibr\'e est exceptionnel;
\item
il existe exactement deux sections $s_1$, $s_2$ d'auto-intersection~$-n$, qui sont disjointes; de plus chaque fibre exceptionnelle a une composante qui touche $s_1$ et l'autre qui touche $s_2$;
\item
il existe un morphisme birationnel $\eta_n:S\rightarrow \mathbb{F}_n$ de fibr\'es en coniques qui est l'\'eclatement de $2n$ points, tous appartenant \`a une m\^eme section d'auto-intersection~$n$;
\item
il existe un morphisme birationnel $\eta_0:S\rightarrow \mathbb{F}_0$ de fibr\'es en coniques qui est l'\'eclatement de $2n$ points, dont $n$ sont sur une section d'auto-intersection~$0$ et les $n$ autres sur une autre section d'auto-intersection~$0$.
\end{enumerate}
\end{lemm}
\begin{proof}
Supposons que le fibr\'e est exceptionnel et notons $s_1$ une des sections d'auto-intersection~$-n$ et $s_2$ l'autre. En contractant dans chaque fibre exceptionnelle la composante qui ne touche pas $s_1$ nous obtenons le morphisme $\eta_n$. Ceci d\'emontre en particulier l'assertion $(2)$ et l'assertion $(3)$. En contractant $n$ composantes qui touchent $s_1$ et $n$ composantes qui touchent $s_2$, nous trouvons le morphisme $\eta_0$ (et donc l'assertion $(4)$).
La d\'emonstration se termine en v\'erifiant que chacune des assertions $(2),(3),(4)$ implique directement que le fibr\'e est exceptionnel. 
\qed\end{proof}
\begin{coro}\label{Coro:IsoFibCExcp}
Deux fibr\'es en coniques exceptionnels sont isomorphes si et seulement les points de $\mathbb{P}^1$ sur lesquels ils sont singuliers sont les m\^emes, \`a automorphisme de $\mathbb{P}^1$ pr\`es.
\end{coro}
\begin{proof}
Il s'agit de montrer qu'un automorphisme de $\mathbb{P}^1$ se rel\`eve \`a un automorphismes de fibr\'es en coniques. Utilisons le morphisme $\eta_n$ du lemme~\ref{Lem:EquiFibExc}. Un automorphisme de $\mathbb{P}^1$ peut \^etre vu comme un automorphisme de la section de $\mathbb{F}_n$ d'auto-intersection~$n$ et s'\'etend \`a un automorphisme de $\mathbb{F}_n$, qui se rel\`eve, via $\eta_n$ \`a un automorphisme du fibr\'e en coniques.
\qed\end{proof}
Le lemme suivant est une cons\'equence de \cite[Proposition 5.3]{bib:DoI}.
\begin{lemm}\label{Lem:AutExcGDelta}
Soit $(S,\pi)$ un fibr\'e en coniques exceptionnel et notons $\Delta\subset \mathbb{P}^1$ l'ensemble des points qui ont une fibre singuli\`ere. Alors, les assertions suivantes sont satisfaites:
\begin{enumerate}
\item
l'action de $\Aut(S,\pi)$ sur la base de la fibration donne lieu a une suite exacte 
\[1\rightarrow \Aut(S/\mathbb{P}^1)\rightarrow \Aut(S,\pi) \rightarrow H_{\Delta} \rightarrow 1,\]
o\`u $H_{\Delta}$ est le sous-groupe de $\PGL(2,\C)$ constitu\'e des \'el\'ements qui pr\'eservent l'ensemble $\Delta$ et $\Aut(S/\mathbb{P}^1)\cong \C^{*}\rtimes\Z{2}$.
\item
Chaque \'el\'ement non-trivial de $\C^{*}\subset \Aut(S/\mathbb{P}^1)$ fixe (point par point) deux courbes rationnelles et chaque \'el\'ement de $\Aut(S/\mathbb{P}^1)\backslash \C^{*}$ est une involution qui fixe une courbe irr\'eductible, rev\^etement double de $\mathbb{P}^1$, via $\pi$, ramifi\'ee aux points de~$\Delta$.
\item
Le groupe $\C^{*}\subset \Aut(S/\mathbb{P}^1)$ agit trivialement sur $\Pic{S}$.
\item
Pour toute involution $\sigma \in \Aut(S/\mathbb{P}^1)\backslash \C^{*}$, on a \[\rkPic{S}^{\sigma}=\rkPic{S}^{\Aut(S,\pi)}=2\].
 \end{enumerate}
\end{lemm}
\begin{proof}
Soit $\Delta \subset \mathbb{P}^1$ l'ensemble des points ayant une fibre singuli\`ere et soit $2n$ le nombre de points de $\Delta$. Notons \'egalement $s_1$ et $s_2$ les deux sections d'auto-intersection~$-n$ et 
$\eta_n:S\rightarrow \mathbb{F}_n$, $\eta_0:S\rightarrow \mathbb{F}_0$ les morphismes donn\'es par le lemme~\ref{Lem:EquiFibExc}.
Le groupe $\Aut(S,\pi)$ agit sur l'ensemble $\{s_1,s_2\}$ et donne donc lieu \`a une suite exacte 
\[1\rightarrow H \rightarrow \Aut(S,\pi) \rightarrow W\rightarrow 1,\]
o\`u $W\subset \Z{2}$. De plus, pour $i=0$ et $i=n$ le groupe $H$ est conjugu\'e par $\eta_i$ au groupe des automorphismes de $\mathbb{F}_i$ qui fixent les points \'eclat\'es par $\eta_i$.

Montrons que $W=\Z{2}$. En \'ecrivant $\Delta=\{p_1\}_{i=1}^{2n}$, nous pouvons supposer que les points \'eclat\'es par le morphisme $\eta_0:S\rightarrow \mathbb{P}^1\times\mathbb{P}^1$ sont $\{(p_i,(0:1))\}_{i=1}^n,\{(p_i,(1:0))\}_{i=1}^{2n}$. 
Choisissons pour $i=1,...,2n$ une forme homog\`ene de degr\'e~$1$ s'annulant en $p_i$, que l'on note $m_i$, et d\'efinissons alors l'application birationnelle de $\mathbb{P}^1\times\mathbb{P}^1$ suivante:
\[\varphi:\left((x_1:x_2),(y_1:y_2)\right)\dasharrow \left((x_1:x_2),(y_2\prod_{i=1}^n m_{i}(x_1,x_2):y_1\prod_{i=n+1}^{2n} m_{i}(x_1,x_2))\right).\]
Les points bases de $\varphi$ et de son inverse (qui est $\varphi$ elle-m\^eme) \'etant exactement les points \'eclat\'es par $\eta_0$, $\varphi$ se rel\`eve \`a une involution $\eta_0^{-1}\varphi\eta_0\in \Aut(S)$ qui pr\'eserve la fibration, \'echange $s_1$ sur $s_2$ et \'echange donc les composantes de chaque fibre singuli\`ere. Ceci d\'emontre que $W=\Z{2}$ et que $\Aut(S,\pi)=H\rtimes <\varphi>$.

Observons que le noyau de l'action de $H$ sur la base la fibration est le relev\'e du tore de $\Aut(\mathbb{P}^1\times\mathbb{P}^1)$ constitu\'e des automorphismes de la forme
\[\left((x_1:x_2),(y_1:y_2)\right)\mapsto \left((x_1:x_2),(y_1:\alpha y_2)\right), \alpha \in \C^{*}.\]
Le groupe $\Aut(S/\mathbb{P}^1)$  est donc le groupe isomorphe \`a $\C^{*}\rtimes \Z{2}$, engendr\'e par ce tore et par $\varphi$. Les \'el\'ements non-triviaux du tore fixent deux courbes rationnelles et les autres \'el\'ements non triviaux de $\Aut(S/\mathbb{P}^1)$ fixent des courbes irr\'eductibles, qui sont des rev\^etements double de $\mathbb{P}^1$, via $\pi$, ramifi\'es le long des points de $\Delta$. Le tore laisse chaque composante de chaque fibre singuli\`ere invariante et donc agit trivialement sur $\Pic{S}$.

Le groupe $\Aut(S,\pi)$ agissant sur l'ensemble des fibres exceptionnelles, l'image de l'action de $\Aut(S,\pi)$ sur la base de la fibration est contenue dans le groupe $H_{\Delta}$ des automorphismes de $\mathbb{P}^1$ qui pr\'eservent l'ensemble $\Delta$.
R\'eciproquement, tout \'el\'ement de ce type peut se voir comme un automorphisme de la section de $\mathbb{F}_n$ qui contient les $2n$ points \'eclat\'es par $\eta_n$. Un tel automorphisme s'\'etend \`a un automorphisme de $\mathbb{F}_n$ et se rel\`eve donc \`a un \'el\'ement de $H$.

Il reste \`a prouver la derni\`ere assertion. Le groupe $\Pic{S}$ est engendr\'e par le diviseur canonique $K_S$, le diviseur $f$ d'une fibre et les diviseurs $E_1,...,E_{2n}$ des courbes contract\'ees par $\eta_n$. Observons que $\sigma$ \'echange les composantes de chaque fibre singuli\`ere et donc envoie $E_i$ sur $f-E_i$. La matrice de cet automorphisme relativement \`a la base $K_S,f,f-2E_1,...,f-2E_{2n}$ est donc diagonale, avec deux valeurs propres \'egales \`a $1$ et les autres \`a $-1$. Ceci montre que la partie de $\Pic{S}$ invariante par $\sigma$ (et donc par $\Aut(S,\pi)$) est de rang $2$, engendr\'e par $f$ et $K_S$.
\qed\end{proof}

\begin{prop}\label{Prp:AutSpiAutSfibE}
Soit $(S,\pi)$ un fibr\'e exceptionnel. Le groupe $\Aut(S)$ est un sous-groupe alg\'ebrique du groupe de Cremona. De plus, $\Aut(S)=\Aut(S,\pi)$ si et seulement si le nombre de fibres singuli\`eres de $\pi$ est plus grand ou \'egal \`a $4$.
\end{prop}
\begin{proof}
L'alg\'ebricit\'e du groupe $\Aut(S,\pi)$ suit de la proposition~\ref{Prp:AutFibMestAlgebrique}. De plus, le noyau de l'action de $\Aut(S)$ sur $\Pic{S}$ contient un groupe isomorphe \`a $\C^{*}$ (lemme~\ref{Lem:AutExcGDelta}). Ceci implique \cite{bib:Har} que le groupe $\Aut(S)$ agit de mani\`ere finie sur $\Pic{S}$ et est donc un sous-groupe alg\'ebrique du groupe de Cremona.

Si le nombre de fibres singuli\`eres est $2$, alors d'apr\`es le lemme~\ref{Lem:EquiFibExc}, la surface $S$ est obtenue en \'eclatant deux points g\'en\'eraux de $\mathbb{P}^1\times\mathbb{P}^1$ et est donc la surface de del Pezzo de degr\'e~$6$. Ceci montre que $\Aut(S)$ est strictement plus grand que $\Aut(S,\pi)$ (voir paragraphe~\ref{SubSec:dPezzo6}).

R\'eciproquement, si le groupe $\Aut(S)$ est strictement plus grand que $\Aut(S,\pi)$, alors l'\'egalit\'e $\rkPic{S}^{\Aut(S,\pi)}=2$ (lemme~\ref{Lem:AutExcGDelta}) implique que $\rkPic{S}^{\Aut(S)}=1$. Ceci montre que $S$ est une surface de del Pezzo (en utilisant un argument de moyenne) et donc que le nombre de fibres singuli\`eres est au plus $2$ (sinon il existerait une courbe d'auto-intersection~$-n$ avec $n\geq 2$). 
\qed\end{proof}

Nous finissons cette pr\'esentation des fibr\'es exceptionnels par un lemme technique sur ceux-ci, qui reprend \cite[Proposition 5.2]{bib:DoI} (ce dernier r\'esultat est prouv\'e pour un groupe fini, mais la preuve s'ajuste naturellement au cas g\'en\'eral).

\begin{lemm}\label{Lem:ContenuDansExc} 
Soit $(S,\pi)$ un fibr\'e en coniques avec au moins une fibre singuli\`ere dont les composantes sont permut\'ees par un \'el\'ement de $\Aut(S,\pi)$.

Si $\Aut(S/\mathbb{P}^1)$ contient un \'el\'ement non trivial qui agit trivialement sur $\Pic{S}$ alors $S$ est un fibr\'e exceptionnel. Sinon, $\Aut(S/\mathbb{P}^1)$ est isomorphe \`a $(\Z{2})^r$ pour $r=0,1,2$.
\end{lemm}
\begin{proof}
Soit $G\subset \Aut(S/\mathbb{P}^1)$ le noyau de l'action de $\Aut(S,\pi)$ sur $\Pic{S}$.  En contractant une composante dans chaque fibre singuli\`ere, nous obtenons un morphisme birationnel de fibr\'es en coniques $G$-\'equivariant $\eta:S\rightarrow \mathbb{F}_r$, pour un certain $r\geq 0$. De plus  $R=\eta G \eta^{-1}\subset \Aut(\mathbb{F}_r)$ fixe tous les points \'eclat\'es par $\eta$. Si $R\subset \PGL(2,\C(x))$ n'est pas trivial, alors  il fixe un ou deux points sur chaque fibre de $\mathbb{F}_n$ et fixe donc une ou deux sections de $\mathbb{F}_r$. Le relev\'e de l'ensemble de ces sections sur $S$ est invariant par $\Aut(S,\pi)$. Le fait qu'un \'el\'ement de $\Aut(S,\pi)$ \'echange les composantes d'une fibre singuli\`ere implique qu'il y a exactement 2 sections $s_1$ et $s_2$ disjointes, fix\'ees par $G$ et \'echang\'ees par un \'el\'ement de $\Aut(S,\pi)$; elles ont donc la m\^eme auto-intersection~$-n$. En contractant dans chaque fibre exceptionnelle la composante qui touche $s_1$, on obtient le morphisme birationnel $\eta_n:S\rightarrow \mathbb{F}_n$ du lemme~\ref{Lem:EquiFibExc}, ce qui implique que le fibr\'e $(S,\pi)$ est exceptionnel.

Supposons maintenant que $G$ est trivial. Ceci implique que tout \'el\'ement non-trivial $\alpha\in \Aut(S/\mathbb{P}^1)$ est une involution et la conclusion suit du fait que $\Aut(S/\mathbb{P}^1)\subset \PGL(2,\C(x))$.
\qed\end{proof}

\subsection{$(\Z{2})^2$-fibr\'es en coniques}
Le lemme suivant se d\'emontre directement par calcul dans $\PGL(2,\C(x))\subset \Jon$; nous le laissons en exercice.
\begin{lemm}\label{Lem:PGL2Cx}
Toute involution de $\PGL(2,\C(x))$ est conjugu\'ee \`a un \'el\'ement de la forme $\sigma_f=\left(\begin{array}{cc}0 & f \\ 1 & 0\end{array}\right)$, o\`u $f\in \C(x)^{*}$. De plus, $\sigma_f$ et $\sigma_g$ sont conjugu\'es si et seulement si $f/g$ est un carr\'e de $\C(x)^{*}$.\qed
\end{lemm}
Ceci implique notamment que deux involutions de $\PGL(2,\C(x))$ sont conjugu\'ees si et seulement si elles ont le m\^eme d\'eterminant (ce dernier \'etant \`a valeur dans $\C(x)^{*}/(\C(x)^{*})^2$). 

\begin{defi}\label{Def:xi}
On notera $\xi$ l'application bijective qui associe \`a un ensemble $A=\{(a_i:b_i)\}_{i=1}^{2n}\subset \mathbb{P}^1$ de $2n\geq 0$ points du plan la classe de $\prod_{i=1}^{2n} (a_ix-b_i)$  dans $\C(x)^{*}/(\C(x)^{*})^2$.
\end{defi}
Rappelons que l'on voit $\PGL(2,\C(x))$ comme un sous-groupe de $\Jon$ (definition \ref{Def:Jonq}), et que l'\'el\'ement $\left(\begin{array}{cc}\alpha& \beta\\ \gamma &\delta \end{array}\right)\in\PGL(2,\C(x))$ s'interpr\`ete alors comme l'application birationelle
$(x,y)\dasharrow \left(x,\frac{\alpha(x)y+\beta(x)}{\gamma(x)+\delta(x)}\right)$ de $\C^2$ (ou $\mathbb{P}^1\times \mathbb{P}^1$).
Les deux lemmes suivants se v\'erifient \'egalement directement \`a la main.
\begin{lemm}
Soit $\sigma \in \PGL(2,\C(x))$ une involution, de d\'eterminant $\det(\sigma)\in \C(x)^{*}/(\C(x)^{*})^2$.

Si $\det(\sigma)=[1]$, alors $\sigma$ est diagonalisable et fixe deux courbes rationnelles.

Si $\det(\sigma)\not=[1]$ alors $\sigma$ n'est pas diagonalisable et  fixe une courbe irr\'eductible, birationnelle \`a un rev\^etement double de $\mathbb{P}^1$, ramifi\'e le long de l'ensemble $A$ ayant un nombre pair positif de points, tel que $\xi(A)=\det(\sigma)$.\qed
\end{lemm}

\begin{lemm}\label{Lem:Nsigma}
Soit $\sigma=\left(\begin{array}{cc}0 & f \\ 1 & 0\end{array}\right)\in \PGL(2,\C(x))$, o\`u $f\in \C(x)^{*}$ n'est pas un carr\'e et notons $N_{\sigma}$ le normalisateur de $<\sigma>$ dans le groupe $\PGL(2,\C(x))$.

Le groupe $N_{\sigma}$ est constitu\'e des matrices de la forme $\left(\begin{array}{cc}a & bf \\ b  & a\end{array}\right)$ ou $\left(\begin{array}{cc}a & -bf \\ b  & -a\end{array}\right)$, avec $a,b \in \C(x)$ et est donc \'egal \`a $N_{\sigma}^0\rtimes \Z{2}$, o\`u $N_{\sigma}^0$ est le groupe isomorphe \`a $\C(x)[\sqrt{f}]^{*}/\C(x)^{*}$, via l'homomorphisme $\left(\begin{array}{cc}a & bf \\ b  & a\end{array}\right)\mapsto [a+b\sqrt{f}]$ et $\Z{2}$ est engendr\'e par l'involution diagonale. L'action de $\Z{2}$ sur $N_{\sigma}^{0}$ envoie $[\rho]=[a+b\sqrt{f}]$ sur $[\rho^{-1}]=[a-b\sqrt{f}]$.\qed
\end{lemm}

Rappelons le r\'esultat suivant, d\'emontr\'e dans \cite[Lemme 2.1]{bib:BeaPel}.
\begin{lemm}\label{Lem:BeaV}
L'application $\det:\PGL(2,\C(x))\rightarrow \C(x)^{*}/(\C(x)^{*})^2$ induit  une bijection entre les classes de conjugaison de sous-groupes de $\PGL(2,\C(x))$ isomorphes \`a $(\Z{2})^2$ et les sous-groupes de  $\C(x)^{*}/(\C(x)^{*})^2$ d'ordre $\leq 4$.\qed
\end{lemm}
Nous avons introduit en (\ref{SubSec:Z22fibr}) la notion de ${(\Z{2})^2}$-fibr\'es en coniques. Nous avons \'egalement d\'efini pour chaque  $(\Z{2})^2$-fibr\'e en coniques son \emph{triplet de points de ramification}, qui est un ensemble de trois ensembles de points de $\mathbb{P}^1$. D\'ecrivons maintenant l'importance de ces points.

\begin{prop}\label{Prp:Z2FibRam}
Soit  $\pi:S\rightarrow \mathbb{P}^1$ un $(\Z{2})^2$-fibr\'e en coniques avec $k$ fibres singuli\`eres, soit $\sigma_1,\sigma_2,\sigma_3$ les trois involutions de $\Aut(S/\mathbb{P}^1)$ et soit $\{A_1,A_2,A_3\}$ le triplet de points de ramification du fibr\'e, de telle sorte que $\sigma_i$ fixe une courbe $C_i$, rev\^etement double de $\mathbb{P}^1$  via $\pi$, ramifi\'ee au-dessus des points de $A_i$. Notons $2a_i$ le nombre de points de $A_i$, pour $i=1,..,3$. Alors:
\begin{enumerate}
\item
pour $i=1,..,3$, l'involution $\sigma_i$ permute les deux composantes de $2a_i$ fibres singuli\`eres, qui sont les fibres au-dessus des points de $A_i$;
\item
chacune des fibres singuli\`eres de $\pi$ a ses deux composantes \'echang\'ees par deux involutions de $\Aut(S/\mathbb{P}^1)$, en particulier $k=a_1+a_2+a_3$;
\item
en notant $f$ le diviseur d'une fibre de $\pi$, on a \[\Pic{S}^{\Aut(S/\mathbb{P}^1)}=\Pic{S}^{\Aut(S,\pi)}=\mathbb{Z}K_S\oplus \mathbb{Z}f;\]
\item
pour $i=1,..,3$, la courbe $C_i$ est \'egale \`a $-K_S+(a_i-2)f$ dans $\Pic{S}$ et $(C_i)^2=4a_i-k$.
\item
si $(S',\pi')$ est un fibr\'e en coniques et s'il existe une application birationnelle de fibr\'es en coniques $\varphi:S\dasharrow S'$ qui conjugue $\Aut(S/\mathbb{P}^1)$ \`a un sous-groupe de $\Aut(S',\pi')$, alors $\varphi$ est un isomorphisme;
\item
l'application qui associe \`a un fibr\'e en coniques son triplet de points de ramification est une bijection entre les classes d'isomorphisme de $(\Z{2})^2$-fibr\'es en coniques et les triplets de ramification, \`a action de $\Aut(\mathbb{P}^1)$ pr\`es;
\end{enumerate}
\end{prop}
\begin{proof}
Notons $V=\Aut(S/\mathbb{P}^1)\cong (\Z{2})^2$. Chaque fibre singuli\`ere de $\pi$ est permut\'ee par au moins une involution de $V$ (et donc par exactement deux involutions). Sinon, l'action de $V$ sur chacune des composantes devant avoir un point fixe (le point singulier de la fibre), une involution $\sigma_i\subset V$ fixerait chacune des deux composantes, et aurait un ensemble de point fixes singulier. 
Observons ensuite que la courbe fix\'ee par $\sigma_i$ est ramifi\'ee exactement aux points correspondant aux fibres singuli\`eres dont les composantes sont permut\'ees par $\sigma_i$, ce qui donne les deux premi\`eres assertions.

Choisissons une composante dans chacune des $k$ fibres singuli\`eres et notons $E_1,...,E_k$ ces courbes exceptionnelles. Alors, l'action de $\Aut(S,\pi)$ sur $\Pic{S}$ fixe le sous-ensemble $\mathbb{Z}K_S\oplus \mathbb{Z} f$ et $V$ agit de mani\`ere diagonale sur $\Pic{S}\otimes\mathbb{Q}$, selon la base $K_S, f, 2E_1-f,...,2E_k-f$. Le fait que pour tout $i$ l'\'el\'ement $2E_i-f$ soit envoy\'e sur $-2E_i+f$ par un \'el\'ement de $V$ implique que $\Pic{S}^{V}=\Pic{S}^{\Aut(S,\pi)}=(\mathbb{Q}K_S\oplus \mathbb{Q} f)\cap \Pic{S}$; or ce dernier groupe est \'egal \`a $\mathbb{Z}K_S\oplus\mathbb{Z}f$ (lemme~\ref{Lem:ZKZf}).  

Soit $i\in\{1,2,3\}$. Comme $C_i$ est invariante par $V$, elle est \'egale \`a $aK_S+bf$ pour $a,b\in \mathbb{Z}$. L'\'egalit\'e $C_i\cdot f=2$ donne $a=-1$ et $C(C+K_S)=-2+2 g(C_i)=-2+2(a_i-1)$ donne $b=2-a_i$. La valeur de $(C_i)^2$ se calcule avec l'\'egalit\'e $(K_S)^2=8-k$.

Prouvons maintenant l'assertion (5). Comme  $\varphi$ est une application birationnelle de fibr\'es en coniques $V$-\'equivariante, elle envoie une fibre sur une fibre. Le groupe $V$ permute les composantes de chaque fibre singuli\`ere; $\varphi$ est donc une composition de liens \'el\'ementaires $V$-\'equivariants (voir notamment \cite{bib:IskSarkisov}, \cite{bib:DoI}), chacun consistant en un \'eclatement d'une orbite de $V$, dont tous les points appartiennent \`a des fibres lisses distinctes, suivi par une contraction des fibres contenant ces points. Le fait que $V$ agisse sans point fixe sur chaque fibre lisse montre que $\varphi$ est un isomorphisme. Ceci d\'emontre l'assertion (5), et \'egalement que l'application de l'assertion (6) est injective. L'utilisation du lemme~\ref{Lem:BeaV} et de la bijection $\xi$ de la d\'efinition~\ref{Def:xi} donne la surjectivit\'e de cette application.
\qed\end{proof}

Comme nous l'avons annonc\'e au d\'ebut de l'article, nous verrons 
que les automorphismes des $(\Z{2})^2$-fibr\'e en coniques sont des sous-groupes alg\'ebriques maximaux du groupe de Cremona si et seulement si la surface ambiante n'est pas une surface de del Pezzo. La proposition suivante d\'etermine quand c'est le cas (ce sera pr\'ecis\'e \`a la section~\ref{Sec:Z2Max}).

\begin{prop}\label{Prp:estdPezzofibres}
Soit  $\pi:S\rightarrow \mathbb{P}^1$ un $(\Z{2})^2$-fibr\'e en coniques avec $k$ fibres singuli\`eres.
\begin{enumerate}
\item
si $k\geq 8$, alors  $S$ n'est pas une surface de del Pezzo;
\item
si $k\leq 5$, alors  $S$ est une surface de del Pezzo;
\item
si $k=6,7$ alors  $S$ n'est pas une surface de del Pezzo, si et seulement si une des situations suivantes se pr\'esente:
\begin{enumerate}
\item
 $S$ contient $4$ sections d'auto-intersection~$-2$;
\item
une courbe fix\'ee par une involution de $\Aut(S,\pi)$ est rationnelle, ce qui implique qu'elle a auto-intersection~$4-k$.
\end{enumerate}
\end{enumerate}
\end{prop}
\begin{rema}
L'assertion (2) suit de \cite[Proposition~5.5]{bib:DoI}.
Le cas o\`u $k=6, 7$ peut donner des surfaces de del Pezzo, mais \'egalement d'autres surfaces plus exotiques (voir section~\ref{Sec:Z2PasDel}).
\end{rema}
\begin{proof}
La premi\`ere assertion suit de l'\'egalit\'e $(K_S)^2=8-k$. 
Soient $\sigma_1,\sigma_2,\sigma_3$ les trois involutions de $\Aut(S/\mathbb{P}^1)$ qui permutent les composantes de respectivement $2a_1$, $2a_2$, $2a_3$ fibres singuli\`eres. 
Soit $s_0$ une section d'auto-intersection~$-r$, o\`u $-r\leq -1$ est le minimum des auto-intersections des sections de $\pi$ et notons $s_i=\sigma(s_0)$ pour $i=1,...,3$.  En contractant les composantes des fibres exceptionnelles ne touchant pas $s_0$ nous obtenons un morphisme  birationnel de fibr\'es en coniques $\eta_r:S\rightarrow\mathbb{F}_r$, qui envoie $s_0$ sur la section exceptionnelle. De plus, pour $i=1,..,3$, la courbe $s_i$ est envoy\'ee sur une section d'auto-intersection~$-r+2a_i$, ce qui montre que $a_i\geq r$.

Supposons que $r=1$, $k\leq 7$ et que $S$ n'est pas une surface de del Pezzo. Ceci est \'equivalent \`a ce qu'il existe une courbe irr\'eductible d'auto-intersection~$\leq -2$ sur $S$ (se d\'eduit par exemple de \cite[Proposition~2 et Th\'eor\`eme~1]{bib:DemDPezzo}). Contractons la section exceptionnelle de $\mathbb{F}_1$ sur le point $q\in \Pn$. La surface $S$ est obtenue en \'eclatant $q$, ainsi que $k$ autres points de $\Pn$, tous appartenant \`a des droites diff\'erentes passant par $q$. L'in\'egalit\'e $k\leq 7$  implique que les courbes d'auto-intersection~$\leq -2$ sont les transform\'ees strictes de droites, coniques ou cubiques de $\Pn$ et le fait que $r=1$ implique que ces courbes intersectent une fibre en $2$ points au moins. Ce sont donc soit des coniques passant par au minimum $6$ points \'eclat\'es (et pas par $q$), soit des cubiques passant par $8$ points \'eclat\'es et ayant un point double \`a l'un des points \'eclat\'es (et pas en $q$). À cause de l'intersection entre de telles courbes, il ne peut y en avoir $2$ et donc la seule courbe d'auto-intersection~$\leq -2$ sur $S$ est laiss\'ee invariante par $\Aut^{0}(S,\pi)$. C'est donc une bisection, qui est fix\'ee par une involution $\sigma_i\in \Aut^{0}(S,\pi)$. Or cette courbe est d'auto-intersection~$4a_i-k$ (proposition~\ref{Prp:Z2FibRam}), ce qui implique que $a_i=1$ et $k\geq 6$ et donne la derni\`ere possibilit\'e cit\'ee dans l'\'enonc\'e. R\'eciproquement, si $k=6,7$ et $a_i=1$ pour un certain   $i$, la surface n'est pas de del Pezzo.

Supposons que $r>1$. Alors, il existe une section d'auto-intersection~$-r\leq -2$ et donc $S$ n'est pas une surface de del Pezzo.  Nous avons observ\'e pr\'ec\'edemment que pour $i=1,..,3$ on a $a_i\geq r$, et comme $a_1+a_2+a_3=k$ (proposition~\ref{Prp:Z2FibRam}), on trouve que $r=2$ et $k\geq 6$. La section est donc d'auto-intersection~$-2$ et son orbite par $V$ donne $4$ sections de m\^eme auto-intersection.\qed\end{proof}

\section{Chaque groupe alg\'ebrique est contenu dans un des groupes de la liste}\label{Sec:Contenu}
Dans cette section, nous d\'emontrons (\`a la proposition~\ref{Prp:Contenu}) que chaque sous-groupe alg\'ebrique du groupe de Cremona est contenu dans un des groupes d\'ecrits dans le th\'eor\`eme~\ref{Theo:Class}.
À l'aide de la proposition~\ref{Prp:Isk} on se ram\`ene aux cas o\`u le groupe agit de mani\`ere minimale sur une surface de del Pezzo ou un fibr\'e en coniques. Ces deux cas divisent la pr\'esente section. 
\subsection{Le cas des automorphismes de surfaces de del Pezzo}
\begin{prop}\label{Prp:ContenuDelPezzo}
Soit $S$ une surface de del Pezzo. Le groupe $\Aut(S)$ est birationnellement conjugu\'e \`a un sous-groupe de l'un des groupes d\'ecrits dans le th\'eor\`eme~\ref{Theo:Class}.
\end{prop}
\begin{proof}
Si $S$ est l'une des surfaces $\Pn$, $\mathbb{P}^1\times\mathbb{P}^1$ ou une surface de del Pezzo de degr\'e~$1,4,5$ ou $6$, alors $\Aut(S)$ fait partie de la liste des groupes du th\'eor\`eme. 
Si $S$ est l'\'eclatement de $1$ (respectivement $2$) point(s) dans le plan, alors le couple $(\Aut(S),S)$ n'est pas minimal; on contracte une courbe et le groupe $\Aut(S)$ est birationnellement conjugu\'e \`a un sous-groupe de $\Aut(\Pn)$ (respectivement $\Aut(\mathbb{P}^1\times\mathbb{P}^1)$). 
Si $S$ est une surface de del Pezzo de degr\'e~$2$ (respectivement $3$) et que $\Aut(S)$ fixe un point de $S$ qui n'est pas sur une courbe exceptionnelle, l'\'eclatement de ce point conjugue $\Aut(S)$ \`a un sous-groupe de $\Aut(S')$, pour une surface $S'$ de del Pezzo de degr\'e~$1$ (respectivement $2$). Le dernier cas est celui o\`u le degr\'e de $S$ est $2$ ou $3$ et o\`u tous les points fixes de l'action de $\Aut(S)$ sur $S$ sont sur les courbes exceptionnelles, qui fait partie de la liste du th\'eor\`eme.
\qed\end{proof}
\subsection{Le cas des automorphismes de fibr\'es en coniques}
Rappelons tout d'abord le r\'esultat suivant, probablement connu du sp\'ecialiste.
\begin{lemm}\label{Lem:1ou2fibresSing}
Soit $(S,\pi)$ un fibr\'e en coniques et supposons que le nombre de fibres singuli\`eres est $1$ ou $2$ et que le couple $(\Aut(S,\pi),S)$ est minimal. Alors, $S$ est la surface de del Pezzo de degr\'e~$6$ (en particulier le nombre de fibres singuli\`eres est~$2$).
\end{lemm}
\begin{proof}
Notons $-r$ l'auto-intersection la plus petite de toutes les sections de $\pi$. En contractant une composante dans chaque fibre singuli\`ere, on v\'erifie que $r$ existe et $r>0$. 
La minimalit\'e du couple $(\Aut(S,\pi),S)$ implique qu'il existe un automorphisme qui \'echange les composantes de chaque fibre singuli\`ere et donc qu'il existe au moins $2$ sections d'auto-intersection~$-r$; on note $s_1$ et $s_2$ ces deux sections. En contractant dans chaque fibre singuli\`ere la composante qui ne touche pas $s_1$ on obtient un morphisme birationnel de fibr\'es en coniques $\pi:S\rightarrow \mathbb{F}_r$ qui envoie $s_1$ sur la section exceptionnelle et envoie $s_2$ sur une section d'auto-intersection~$-r+l$, o\`u $l\in\{1,2\}$. Ceci implique que $-r+l\geq r$ et donc que $r=1$ et $l=2$. La surface $S$ est donc obtenue en \'eclatant deux points de $\mathbb{F}_1$ dans des fibres diff\'erentes et pas sur la section exceptionnelle, elle est alors isomorphe \`a la surface de del Pezzo de degr\'e~$6$.
\qed\end{proof}

Nous pouvons maintenant d\'emontrer la proposition cl\'e de cette section, en utilisant la structure alg\'ebrique du groupe $\Jon=\PGL(2,\C(x))\rtimes \PGL(2,\C)$.
\begin{prop}\label{Prp:CleJonq}
Soit $G$ un sous-groupe alg\'ebrique du groupe de de Jonqui\`eres et notons $G'\subset G$ et $H\subset \PGL(2,\C)$ le noyau et l'image de l'action de $G$ sur la base de la fibration. On suppose que $H$ est fini, alors:
\begin{enumerate}
\item
Si $G'=\{1\}$, alors $G$ est conjugu\'e -- dans $\Jon$ -- \`a $H$.
\item
Si $G'\cong \Z{2}$ est engendr\'e par une involution dont le d\'eterminant n'est pas trivial, alors 
$G$ normalise un groupe $V\subset \PGL(2,\C(x))$ isomorphe \`a $(\Z{2})^2$ et contenant $G'$.
\end{enumerate}
\end{prop}
\begin{proof}
Remarquons tout d'abord que le corps $\C(x)$ a la propri\'et\'e $C_1$  par le th\'eor\`eme de Tsen et donc que pour tout groupe fini $H$, les ensembles $H^1(H,\GL(2,\C(x)))$ et $H^2(H,\C(x)^{*})$ sont r\'eduits \`a un \'el\'ement (voir \cite{bib:Serre}, chapitre X, propositions 3 et 11). Ceci implique donc que $H^1(H,\PGL(2,\C(x))=\{1\}$.

Supposons que $G'=\{1\}$, ce qui implique qu'il existe une section $H\rightarrow G\subset \Jon$. La diff\'erence de cette section avec la section canonique donn\'ee par l'inclusion de $G$ dans $\Jon$ repr\'esente un co-cycle. Le fait que $H^1(H,\PGL(2,\C(x))=\{1\}$ implique que ce co-cycle est conjugu\'e au co-cycle trivial, i.e. que $G$ est conjugu\'e \`a $H$, ce qui d\'emontre l'assertion (1).

Il reste \`a prouver l'assertion (2). Notons $\sigma$ l'\'el\'ement d'ordre $2$ de $G'$. Apr\`es conjugaison par un \'el\'ement de $\PGL(2,\C(x))$, on peut supposer que $\sigma=\left(\begin{array}{cc} 0& f\\ 1 & 0\end{array}\right)$, pour un certain \'el\'ement $f\in \C(x)^{*}$. De plus, l'hypoth\`ese sur le d\'eterminant de $\sigma$ implique que $f$ n'est pas un carr\'e.

Notons $N_{\sigma}$ le normalisateur de $\sigma$ dans le groupe $\PGL(2,\C(x))$ et rappelons (lemme~\ref{Lem:Nsigma}) que $N_{\sigma}=N_{\sigma}^0\rtimes \Z{2}$, o\`u $N_{\sigma}^0$ est isomorphe \`a $\C(x)[\sqrt{f}]^{*}/\C(x)^{*}$ et $\Z{2}$ agit sur $N_{\sigma}^{0}$ en envoyant $[\rho]=[a+b\sqrt{f}]$ sur $[\rho^{-1}]=[a-b\sqrt{f}]$. 

D\'ecrivons maintenant une construction qui associe \`a tout \'el\'ement $h\in H$ un \'el\'ement $\rho_h\in N_{\sigma}^0$ et un \'el\'ement $\mu_h\in \C(x)^{*}$. La pr\'eimage de $h$ dans $G$ est constitu\'ee de deux \'el\'ements $(\gamma,h),(\sigma\gamma,h)\in \Jon$. Comme $\sigma$ est dans le centre de $G$, on a $\gamma h(\sigma)\gamma^{-1}=\sigma$. L'\'el\'ement $h(\sigma)=\left(\begin{array}{cc} 0& h(f)\\ 1 & 0\end{array}\right)$ \'etant conjugu\'e \`a $\sigma=\left(\begin{array}{cc} 0& f\\ 1 & 0\end{array}\right)$ dans $\PGL(2,\C(x))$, il existe $\mu \in \C(x)^{*}$ tel que $\mu^2=f/h(f)$ (Lemme~\ref{Lem:PGL2Cx}). On \'ecrit $\beta=\left(\begin{array}{cc} \mu& 0\\ 0 & 1\end{array}\right)$, ce qui donne $\beta h(\sigma)\beta^{-1}=\sigma$. Alors, $\gamma$ s'\'ecrit $\gamma=\alpha\beta$, avec $\alpha\in N_{\sigma}$. Choisissons le signe de $\mu$ de telle sorte que $\alpha\in N_{\sigma}^0$. Nous d\'esignons alors respectivement par $\rho_h$ et $\mu_h$ les \'el\'ements $\alpha^2\in N_{\sigma}^0$ et $\mu\in\C(x)^{*}$, en remarquant que ces deux \'el\'ements sont uniquement d\'etermin\'es par $h$, puisque $\sigma\alpha\in N_{\sigma}^0$ et $(\sigma\alpha)^2=\alpha^2$.

\'Etudions maintenant plus pr\'ecis\'ement les deux application $\rho,\mu$ qui envoient $h$ respectivement sur $\rho_h$ et $\mu_h$.
\'Etant donn\'es deux \'el\'ements $h_1,h_2\in H$, posons $h_3=h_1h_2$ et choisissons comme avant $(\alpha_i\beta_i,h_i)\in G$ pour $i=1,2,3$, avec $\alpha_i\in N_{\sigma}^0$ et $\beta_i=\left(\begin{array}{cc} \mu_i& 0\\ 0 & 1\end{array}\right)$, $\mu_i^2=f/h_i(f)$. On peut choisir $\alpha_3\beta_3$ de telle sorte que $(\alpha_1\beta_1,h_1)\cdot (\alpha_2\beta_2,h_2)=(\alpha_3\beta_3,h_3)$, ce qui implique que $\alpha_3=\alpha_1\beta_1\cdot h_1(\alpha_2\beta_2)\cdot (\beta_3)^{-1}$. En \'ecrivant $\alpha_2=\left(\begin{array}{cc}a & bf \\b  &  a\end{array}\right)\in N_{\sigma}^0$, calculons:
\[\beta_1 h_1(\alpha_2)(\beta_1)^{-1}=\left(\begin{array}{cc}\mu_1\cdot h_1(a) &  (\mu_1)^2\cdot h_1(bf) \\ h_1(b)  &  \mu_1 \cdot h_1(a)\end{array}\right)=\left(\begin{array}{cc}\mu_1\cdot h_1(a) & h_1(b)\cdot f \\ h(b)  & \mu_1 \cdot h_1(a)\end{array}\right);\]
ceci montre que $\beta_1 h_1(\alpha_2)(\beta_1)^{-1}$  appartient \`a $N_{\sigma}^0$. De la relation
 \[\alpha_3=\alpha_1\cdot(\beta_1 h_1(\alpha_2)(\beta_1)^{-1}) \cdot \beta_1\cdot h_1(\beta_2)\cdot (\beta_3)^{-1}\] on d\'eduit que l'\'el\'ement diagonal $\beta_1\cdot h_1(\beta_2)\cdot (\beta_3)^{-1}$ appartient aussi \`a $N_{\sigma}^0$ et est alors trivial, d'o\`u l'\'egalit\'e $\mu_3=\mu_1\cdot h_1(\mu_2)$ qui montre que l'application  $\mu:H\rightarrow \C(x)^{*}$ qui envoie $h$ sur $\mu_h$ est un co-cycle. Du fait que $H^1(H,\C(x)^{*})=\{1\}$, il existe $\nu\in \C(x)^{*}$ tel que $\mu_h=\nu/h(\nu)$ pour tout $h\in H$; l'\'el\'ement $f/\nu^2$ est alors invariant par $H$. Apr\`es avoir conjugu\'e $G$ par $\left(\begin{array}{cc} 1& 0\\ 0 & \nu\end{array}\right)$ \emph{on peut donc supposer que $\mu_h=1$ pour tout $h$ -- qui \'equivaut \`a dire que $f$ est invariant par $H$}. En reprenant $h_1,h_2,h_3$ comme avant, on trouve alors (du fait que $\beta_i=1$ pour tout $i$) l'\'egalit\'e $\alpha_3=\alpha_1\cdot h_1(\alpha_2)$, qui montre que l'application  $\rho:H\rightarrow N_{\sigma}^0$ qui envoie $h$ sur $\rho_h$ est maintenant un co-cycle.
 La suite exacte $H$-\'equivariante
 \[1 \rightarrow \C(x)^{*}\rightarrow \C(x)[\sqrt{f}]^{*}\rightarrow N_{\sigma}^0\rightarrow 1,\]
et les \'egalit\'es $H^1(H,\C(x)[\sqrt{f}]^{*})=\{1\}$ et $H^2(H,\C(x)^{*})=\{1\}$ impliquent la trivialit\'e de $H^1(H,N_{\sigma}^0)$. Il existe donc $\gamma \in N_{\sigma}^0$ tel que $\rho_h=\gamma\cdot h(\gamma)^{-1}$ pour tout $h\in H$. Il reste \`a voir que cette condition implique que l'involution $(\gamma,-1)\in N_{\sigma}$ commute avec $G$, ce qui montre que groupe $V\subset \PGL(2,\C(x))$ engendr\'e par $\sigma$ et $(\gamma,-1)$  est normalis\'e par $G$ et qui donne le r\'esultat escompt\'e (car $V$ est  isomorphe \`a $(\Z{2})^2$).

Soit $h\in H$ et choisissons $(\alpha\beta,h)\in G$, avec $\alpha\in N_{\sigma}^0$ et $\beta=\left(\begin{array}{cc} \mu& 0\\ 0 & 1\end{array}\right)$, $\mu^2=f/h(f)$, comme pr\'ec\'edemment. Rappelons que $\mu=1$ et donc que $\beta=1$; alors, la conjugaison de $(\gamma,-1)\in N_{\sigma}$ par $(\alpha, h)$ est le produit $(\alpha,1)\cdot (h(\gamma), -1)\cdot (\alpha^{-1},1)$ dans le groupe $N_{\sigma}=N_{\sigma}^0\rtimes \Z{2}$. Ce produit est \'egal \`a $(\alpha,1)\cdot (h(\gamma)\alpha, -1)=(h(\gamma)\alpha^2,-1)=(h(\gamma)\rho_h,-1)=(\gamma,-1)$, ce qui conclut la preuve.\qed\end{proof}

\begin{prop}\label{Prp:ContenuFibreConiques}
Soit  $\pi:S\rightarrow \mathbb{P}^1$ un fibr\'e en coniques, tel que le couple $(\Aut(S,\pi),S)$ soit minimal.
Alors, le groupe alg\'ebrique $\Aut(S,\pi)$ est birationnellement conjugu\'e \`a un sous-groupe d'un des groupes d\'ecrits dans le th\'eor\`eme~\ref{Theo:Class}.
\end{prop}
\begin{proof}
Supposons tout d'abord que $\pi$ est une fibration en droites. La surface $S$ est donc une surface de Hirzebruch. Si c'est $\mathbb{F}_1$, alors $\Aut(S,\pi)$ est conjugu\'e \`a un sous-groupe de $\Aut(\Pn)$, en contractant la section exceptionnelle. Sinon, $\Aut(\mathbb{F}_n)$ est un des groupes de la liste, ce qui termine la d\'emonstration.

Supposons maintenant que le nombre de fibres singuli\`eres est $1$ ou $2$. Le lemme~\ref{Lem:1ou2fibresSing} implique que $S$ est la surface de del Pezzo de degr\'e~$6$, dont le groupe des automorphismes est pr\'esent dans la liste.

Supposons alors que le nombre de fibres singuli\`eres est au minimum $3$ et notons $H\subset \PGL(2,\C)$ l'image de l'action de $\Aut(S,\pi)$ sur la base de la fibration. Le fait que le nombre de fibres singuli\`eres soit au moins $3$ implique que $H$ est fini. 
Si $\Aut(S/\mathbb{P}^1)$ contient un \'el\'ement non-trivial qui agit trivialement sur $\Pic{S}$, alors $(S,\pi)$ est un fibr\'e exceptionnel (lemme~\ref{Lem:ContenuDansExc}) et a au minimum $4$ fibres singuli\`eres (le nombre de fibres singuli\`eres \'etant pair); le groupe $\Aut(S,\pi)$  fait donc partie de la liste.

On peut supposer que $\Aut(S/\mathbb{P}^1)$ ne contient pas d'\'el\'ement non-trivial qui agit trivialement sur $\Pic{S}$. Le groupe $\Aut(S/\mathbb{P}^1)$ est donc isomorphe \`a $(\Z{2})^r$ pour $r=0,1,2$ (lemme~\ref{Lem:ContenuDansExc}). 
Si $r=0$, alors $\Aut(S,\pi)$ est conjugu\'e \`a $H$ (proposition~\ref{Prp:CleJonq}); comme $H \subset \PGL(2,\C)\subset \Aut(\mathbb{P}^1\times\mathbb{P}^1)$, on a termin\'e.

Le groupe $\Aut(S/\mathbb{P}^1)$ contient donc une ou trois involutions. Chacune de ces involutions agissant non-trivialement sur $\Pic{S}$, elle \'echange les composantes d'un certain nombre de fibres singuli\`eres et ses points fixes contiennent donc une seul courbe irr\'eductible. Ceci implique que le d\'eterminant de chaque involution de $\Aut(S/\mathbb{P}^1)\subset \PGL(2,\C(x))$ est non-trivial.

Si $\Aut(S/\mathbb{P}^1)\cong \Z{2}$, la proposition~\ref{Prp:CleJonq} nous dit que $\Aut(S,\pi)$ normalise un sous-groupe $V\subset \PGL(2,\C(x))$, isomorphe \`a $(\Z{2})^2$, qui contient $\Aut(S/\mathbb{P}^1)$. Alors, $\Aut(S,\pi)$ est contenu dans le groupe engendr\'e par $\Aut(S,\pi)$ et $V$, qui est fini et agit alors sur un fibr\'e en coniques.

Il reste donc le cas o\`u $\Aut(S/\mathbb{P}^1)\cong (\Z{2})^2$ et o\`u l'application $\det: \Aut(S/\mathbb{P}^1)\rightarrow \C(x)^{*}/(\C(x)^{*})^2$ est injective. Le fibr\'e en coniques $(S,\pi)$ est donc un $(\Z{2})^2$-fibr\'e en coniques. Si la surface $S$ est de del Pezzo, alors $\Aut(S,\pi)\subset \Aut(S)$ et le r\'esultat suit de la proposition~\ref{Prp:ContenuDelPezzo}. Sinon, le groupe $\Aut(S,\pi)$ est un des groupes de la liste.
\qed\end{proof}
\subsection{Le r\'esultat}
\begin{prop}\label{Prp:Contenu}
Soit $G$ un sous-groupe alg\'ebrique du groupe de Cremona.  Alors, $G$ est conjugu\'e \`a un sous-groupe d'un des groupes d\'ecrits dans  le th\'eor\`eme~\ref{Theo:Class}.\end{prop}
\begin{proof}
On peut supposer que $G$ agit sur une surface $S$ (voir \S\ref{Sec:Res}) et assumer de plus que le couple $(G,S)$ est minimal. D'apr\`es la proposition~\ref{Prp:Isk}, ou bien $S$ est une surface de del Pezzo et le r\'esultat suit de la proposition~\ref{Prp:ContenuDelPezzo}, ou bien $G\subset \Aut(S,\pi)$ pour un certain fibr\'e en coniques $\pi:S\rightarrow \mathbb{P}^1$ et le r\'esultat suit de la proposition~\ref{Prp:ContenuFibreConiques}.
\qed\end{proof}
\section{Minimalit\'e des couples}\label{Sec:MinP}
\setcounter{subsection}{1}
Dans cette section, nous d\'emontrons que les couples du th\'eor\`eme~\ref{Theo:Class} (groupe des automorphismes et surface ambiante) sont minimaux. Ceci servira notamment \`a d\'emontrer la superrigidit\'e birationnelle de celles-ci, \`a la section~\ref{Sec:Superrigidite}.

\begin{lemm}\label{Lem:MinG124}
Soit $(S,\pi)$ un fibr\'e en coniques ayant $k$ fibres singuli\`eres, et supposons qu'il existe un groupe $G\subset \Aut(S,\pi)$ d'ordre $2$ ou $4$, agissant trivialement sur la base de la fibration, et tel que $\rkPic{S}^G=2$. Si $k=4$ ou $k\geq 6$, le couple $(G,S)$ est minimal.
\end{lemm}
\begin{proof}Supposons que $k\geq 1$ et que le couple $(G,S)$ ne soit pas minimal. Il existe donc une orbite de $(-1)$-courbes disjointes $C_1,...,C_l$ avec $l=1,2,4$. En notant $D=\sum_{i=1}^l C_i$ on trouve $D^2=D\cdot K_S=-l$. Comme $D$ est invariant par $G$, on a $D=aK_S+bf$ pour $a,b\in \mathbb{Z}$ (lemme~\ref{Lem:ZKZf}), o\`u $f$ est le diviseur de la fibre de $\pi$; de plus l'in\'egalit\'e $D\cdot f \geq 0$ implique que $a\leq 0$.

Les deux \'equations pr\'ec\'edentes donnent $a(a(K_S)^2-4b)=a(K_S)^2-2b=-l$, ce qui donne $a(l+2b)=l$. Les valeurs possibles pour $l$ donnent pour le triplet $(l,a,b)$ les valeurs $(1,-1,-1)$, $(2,-1,-2)$, $(4,-1,-4)$, $(4,-2,-3)$. L'\'egalit\'e $(K_S)^2=(2b-l)/a$ nous livre respectivement $3$, $6$, $12$ et $5$ comme valeurs possibles pour $(K_S)^2$. Comme $(K_S)^2=8-k$, on trouve que $k$ doit valoir $2,3$ ou~$5$.
\qed\end{proof}
\begin{coro}\label{Coro:MinFibC}
Soit $(S,\pi)$ un fibr\'e en coniques.  Si $(S,\pi)$ est l'un des deux cas suivants, le couple $(\Aut(S,\pi),S)$ est minimal:
\begin{enumerate}
\item
un fibr\'e en coniques exceptionnel ayant au minimum $4$ fibres singuli\`eres;
\item
un ${(\Z{2})^2}$-fibr\'e en coniques, tel que la surface ambiante ne soit pas une surface de del Pezzo.
\end{enumerate}
\end{coro}
\begin{proof}
Dans le premier cas, le lemme~\ref{Lem:AutExcGDelta} implique que $\Aut(S,\pi)$ contient un sous-groupe $G\cong \Z{2}$ tel que   $\rkPic{S}^{G}=2$. Le nombre de fibres singuli\`eres de $(S,\pi)$ \'etant \'egal \`a $2n$ pour $n\geq 2$, le lemme~\ref{Lem:MinG124} implique que le couple $(G,S)$ est minimal.
Dans le second cas, la proposition~\ref{Prp:Z2FibRam} montre que $\Aut(S,\pi)$ contient un sous-groupe $V\cong (\Z{2})^2$ tel que   $\rkPic{S}^{V}=2$. La proposition~\ref{Prp:estdPezzofibres} nous dit que le nombre de fibres singuli\`eres est plus grand ou \'egal \`a $6$; en appliquant \`a nouveau le lemme~\ref{Lem:MinG124} on trouve que le couple $(V,S)$ est minimal.
\qed\end{proof}

\begin{prop}\label{Prp:MinPaires}
Soit $(S,\pi)$ une des fibrations du th\'eor\`eme~\ref{Theo:Class}, et $G=\Aut(S,\pi)$.

Alors, $\pi$ est une $G$-fibration de Mori (au sens de la d\'efinition~\ref{defi:GfibreMori}) et le couple $(G,S)$ est minimal.
\end{prop}
\begin{proof}
Supposons d'abord que $Y\cong \mathbb{P}^1$, et donc que $\pi$ est une fibration en coniques. On a toujours $-K_S\cdot f=2$, o\`u $f$ est le diviseur d'une fibre. Alors, $\pi$ est une $G$-fibration de Mori si et seulement si $\rkPic{S}^G=2$. Si $S$ est une surface de Hirzebruch $\mathbb{F}_n$ pour $n\geq 2$, alors celle-ci est minimale et donc le couple $(\Aut(S,\pi),S)$ \'egalement; comme $\rkPic{S}=2$, on trouve $\rkPic{S}^G=2$, ce qui termine la preuve. Si $\pi$ est un fibr\'e exceptionnel ou un $(\Z{2})^2$-fibr\'e en coniques, le lemme~\ref{Lem:AutExcGDelta} et la proposition~\ref{Prp:Z2FibRam} livrent l'\'egalit\'e $\rkPic{S}^G=2$; le corollaire~\ref{Coro:MinFibC} montre que le couple $(G,S)$ est minimal.

Supposons que $Y$ soit un point, et donc que $S$ est une surface de del Pezzo et $G=\Aut(S)$. Montrons que le couple $(G,S)$ est minimal et qu'aucun fibr\'e en coniques n'est laiss\'e invariant par $G$, ce qui impliquera (proposition~\ref{Prp:Isk}) que $\rkPic{S}^G=1$ et donc que $\pi$ est une $G$-fibration de Mori. Si $S\cong \Pn,\mathbb{P}^1\times\mathbb{P}^1$, alors $S$ est minimale et donc $(G,S)$ aussi et  aucune fibration n'est laiss\'ee invariante.  Si $S$ est la surface de del Pezzo de degr\'e~$6$ (respectivement $5$), alors $G$ agit transitivement sur les $6$ (respectivement $10$) courbes exceptionnelles, ce qui implique le r\'esultat. Si $S$ est de degr\'e~$4$, alors $G$ contient cinq involutions fixant des courbes elliptiques; donc apr\`es une conjugaison birationnelle, $G$ ne peut agir sur aucune surface de del Pezzo de degr\'e~$\geq 5$, ni sur aucun fibr\'e en coniques. Si $S$ est de degr\'e~$3$, la minimalit\'e du couple $(G,S)$ est impos\'ee par l'\'enonc\'e du th\'eor\`eme~\ref{Theo:Class}, ce qui implique qu'aucun fibr\'e en coniques $\pi:S\rightarrow \mathbb{P}^1$ n'est laiss\'e invariant par $G$ (proposition~\ref{Prp:CubicFibCpasMin}). Si $S$ est de degr\'e~$2$ (respectivement $1$), alors $G$ contient l'involution de Geiser (respectivement celle de Bertini), qui n'agit sur aucune autre surface de del Pezzo de degr\'e sup\'erieur ni sur aucun fibr\'e en coniques, les deux involutions fixant des courbes non-hyperrelliptiques (voir \cite{bib:BaB}).
\qed\end{proof}

\section{Superrigidit\'e birationnelle}\label{Sec:Superrigidite}
\setcounter{subsection}{1}
La proposition~\ref{Prp:MinPaires} montre que chacun des couples $(G,S)$ du th\'eor\`eme~\ref{Theo:Class} est minimal et repr\'esente une $G$-fibration de Mori. Nous d\'emontrons maintenant un r\'esultat plus fort, c'est-\`a-dire que ces $G$-fibrations de Mori sont birationnellement superrigides.

\begin{lemm}\label{Lem:PasDautrefibration}
Soit $\pi:S\rightarrow \mathbb{P}^1$ un fibr\'e en coniques et soit $G\subset \Aut(S,\pi)$ tel que $\rkPic{S}^G=2$. Supposons que $\pi':S\rightarrow \mathbb{P}^1$ est une fibration en coniques avec des fibres diff\'erentes de celles de $\pi$, qui soit invariante par $G$. Alors $S$ est une surface de del Pezzo de degr\'e~$1$, $2$ ou $4$ ou $8$. 
\end{lemm}
\begin{proof}
Si $S$ est une surface de Hirzebruch, c'est forc\'ement $\mathbb{P}^1\times\mathbb{P}^1$, qui est une surface de del Pezzo de degr\'e~$8$.  
On peut alors supposer qu'il y a au moins une fibre singuli\`ere. Notons $f$ le diviseur de la fibre de $\pi$ et $C$ celui de la fibre de $\pi'$. Le lemme~\ref{Lem:ZKZf} implique que $\Pic{S}^G=\mathbb{Z}K_S\oplus \mathbb{Z}f$; on \'ecrit alors $C=-aK_S+bf$ pour $a,b\in \mathbb{Z}$. L'in\'egalit\'e $C\cdot f \geq 0$ implique $a\geq 0$ et les \'equations $C^2=0$ et $-K_S\cdot C=2$ donnent respectivement $a(a(K_S)^2+4b)=0$ et $a(K_S)^2+2b=2$, ce qui implique que $b=-1$ et $a(K_S)^2=4$. On trouve donc que $(K_S)^2\in\{1,2,4\}$. Il reste \`a voir que $S$ est une surface de del Pezzo. Si tel n'est pas le cas, alors il existe une courbe irr\'eductible d'auto-intersection~$\leq -2$ (se d\'eduit par exemple de \cite[Proposition~2 et Th\'eor\`eme~1]{bib:DemDPezzo}); en notant $D$ le diviseur de la courbe on trouve $D\cdot f> 0$ et $D\cdot (-K_S)\leq 0$. Ceci implique l'in\'egalit\'e $D\cdot C<0$, qui est impossible.
\qed\end{proof}

\begin{prop}\label{Prp:Superrigidite}
Soit $\pi:S\rightarrow Y$ une des fibrations du th\'eor\`eme~\ref{Theo:Class} et no\-tons $G=\Aut(S,\pi)$.

Soit $\pi':S'\rightarrow Y'$ une $G$-fibration de Mori et supposons qu'il existe une application birationnelle $\varphi:S\dasharrow S'$ qui soit $G$-\'equivariante.
Alors, $\varphi$ est un isomorphisme de $G$-fibrations de Mori. 

Ceci signifie que la $G$-fibration de Mori $\pi:S\rightarrow Y$ est birationnellement superrigide. 
\end{prop}
\begin{proof}On peut d\'ecomposer $\varphi$ en automorphismes de $G$-fibrations de Mori et \emph{liens \'el\'ementaires} $G$-\'equivariants, chacun passant d'une $G$-fibration de Mori \`a une autre. Ceci est appel\'e actuellement \emph{programme de Sarkisov $G$-\'equiva\-riant} et a \'et\'e d\'emontr\'e dans le cas des surfaces dans \cite{bib:IskSarkisov} (voir aussi \cite{bib:Corti}). Notons que la d\'emonstration de \cite{bib:IskSarkisov} et \cite{bib:Corti} se fait dans le cas o\`u $G$ est le groupe de Galois d'une extension de corps, mais que le cas o\`u $G$ agit de mani\`ere g\'eom\'etrique se d\'emontre de la m\^eme mani\`ere (voir la section~\ref{SubSec:MMpDim2} et l'introduction de \cite{bib:Man}); citons \'egalement la d\'emonstration de \cite{bib:DoI}, faite dans le cadre g\'eom\'etrique. Les liens \'el\'ementaires sont classifi\'es dans \cite{bib:IskSarkisov} (ainsi que dans \cite{bib:DoI}).

Nous allons d\'emontrer qu'aucun lien ne sort de la $G$-surface $(G,S)$, ce qui ach\`evera la d\'emonstration. Il suffit donc de supposer que $\varphi$ est un lien \'el\'ementaire et de trouver une contradiction. Il y a  $4$ types de liens possibles (que l'on retrouve dans \cite{bib:IskSarkisov}), qui d\'ependent notamment de $Y$ et $Y'$; ces derniers pouvant \^etre $\mathbb{P}^1$ ou un point (not\'e $*$). Les $4$ liens sont d\'ecrits dans la figure suivante,
\begin{center}
\begin{tabular}{cccc}
Type $I$& 
Type $II$ & Type $III$& 
Type $IV$\\
$\xymatrix{S\ar[d]_{{\pi}}&S'\ar[l]_{\sigma}\ar[d]^{\pi'}\\
{*}&\mathbb{P}^1\ar[l] } $& 
$\xymatrix{S\ar[d]_{{\pi}}&X\ar[l]_{\sigma}\ar[r]^{\tau}&S'\ar[d]^{\pi'}\\
Y\ar[rr]^{\cong}&&Y' } $ &
$\xymatrix{S\ar[d]_{{\pi}}\ar[r]^{\sigma}&S'\ar[d]^{\pi'}\\
\mathbb{P}^1\ar[r]&{*} } $& 
$\xymatrix{S\ar[d]_{{\pi}}\ar[r]^{\cong}&S'\ar[d]^{\pi'}\\
\mathbb{P}^1&\mathbb{P}^1 } $\\
\end{tabular}
\end{center}
o\`u $G$ agit bir\'eguli\`erement sur $S,S',X$ et $\sigma$ et $\tau$ sont des morphismes birationnels contractant une $G$-orbite et tous les diagrammes sont commutatifs et $G$-\'equivariants.

Si la surface $S$ est $\mathbb{P}^2$, $\mathbb{P}^1\times\mathbb{P}^1$, la surface de del Pezzo de degr\'e~$6$ (\ref{SubSec:dPezzo6}) ou une surface de Hirzebruch $\mathbb{F}_n$ avec $n\geq 2$ alors aucun lien de type $I,II,III$ n'est possible, car il n'existe pas d'orbite finie. Les liens de type $IV$ ne sont pas possibles non plus car aucune de ces surfaces n'admet deux fibrations en coniques $G$-\'equivariantes.

\'Etudions maintenant dans les autres cas chacun des $4$ liens.

\emph{Lien de type $I$} - Dans ce cas,  $S$ est une surface de del Pezzo et $Y$ est un point. De plus, \cite[Th\'eor\`eme 2.6]{bib:IskSarkisov} montre que $(K_S)^2=4,8,9$. Les cas $8$ et $9$ ayant \'et\'e \'etudi\'es pr\'ec\'edemment ($\Pn$ et les surfaces de Hirzebruch), il nous reste $(K_S)^2=4$, mais alors $G=\Aut(S)$ contient un groupe isomorphe \`a $(\Z{2})^4$, engendr\'e par $5$ involutions qui fixent chacune une courbe elliptique (\ref{SubSec:dPdeg4}). Un tel groupe ne pouvant pas agir sur un fibr\'e en coniques, ce lien n'est pas possible.

\emph{Lien de type $II$}. Deux cas se pr\'esentent. Si $Y,Y'\cong \mathbb{P}^1$, alors $\sigma$ \'eclate une orbite de $G$ sur $S$, dont tous les points appartiennent \`a des fibres lisses distinctes et $\tau$ contracte les transform\'ees strictes des fibres de chacun des points \'eclat\'es par $\sigma$. Comme chacun des groupes $G$ du th\'eor\`eme pr\'eservant des fibrations agit sans point fixe sur chaque fibre lisse, ce lien n'est pas possible. Il nous reste le cas o\`u $Y,Y'$ sont des points et $S,S'$ sont des surfaces de del Pezzo. Alors, \cite[Th\'eor\`eme 2.6]{bib:IskSarkisov} montre que le nombre de points \'eclat\'es par $\sigma$ est strictement plus petit que $(K_S)^2$ et que $X$ est une surface de del Pezzo; aucun des point \'eclat\'es n'appartient donc \`a une courbe exceptionnelle. Il suffit alors de voir que les orbites de l'action de $\Aut(S)$ sur $S$ priv\'e de ses courbes exceptionnelles ont toutes une taille au moins \'egale \`a $(K_S)^2$. Si $(K_S)^2=1$, c'est clair; si $(K_S)^2\leq 3$, c'est impos\'e par l'\'enonc\'e du th\'eor\`eme~\ref{Theo:Class} et par le lemme~\ref{Lem:Min3pts}. Si $(K_S)^2=4$, alors $S\subset \mathbb{P}^4$ est l'intersection de deux quadriques (\ref{SubSec:dPdeg4}) et $\Aut(S)$ contient le sous-groupe diagonal $T\cong (\Z{2})^4$ dont les orbites sur $S$ sont toutes de taille au moins \'egal \`a $4$. Si $(K_S)^2=5$, alors $S$ est l'\'eclatement de $4$ points dans le plan et $\Aut(S)$ contient le relev\'e du groupe $\Sym_4\subset \Aut(\Pn)$ pr\'eservant ces points. Or, les orbites de ce dernier sur le plan priv\'e des droites passant par les $4$ points sont toutes de taille au moins \'egale \`a $6$.

\emph{Lien de type $III$}. Ici, $\pi:S\rightarrow Y$ est une fibration en coniques et \cite[Theorem~2.6]{bib:IskSarkisov} nous dit qu'un tel lien n'existe que si $S\cong \mathbb{F}_1$ (ce qui n'est pas le cas ici) ou si $(K_S)^2\in \{3,5,6\}$, c'est-\`a-dire que le nombre de fibres singuli\`eres est $6,3$ ou $2$ et qui n'est pas le cas non plus.

\emph{Lien de type $IV$} Ici, $\pi:S\rightarrow Y$ est une fibration en coniques et le lien consiste en un changement de fibration en coniques sur la m\^eme surface. Il suffit de voir $\Aut(S,\pi)$ ne pr\'eserve aucune autre fibration en coniques sur $S$. Si $S\cong \mathbb{F}_n$ pour $n\geq 2$, c'est vrai car il n'existe qu'une fibration en coniques sur $\mathbb{F}_n$. Si on est dans le cas d'un fibr\'e exceptionnel ou d'un $(\Z{2})^2$-fibr\'e en coniques, alors $S$ n'est pas une surface de del Pezzo (dans le premier cas, la courbe contient deux sections d'auto-intersections $\leq -2$ et dans le deuxi\`eme cas c'est impos\'e par l'\'enonc\'e du th\'eor\`eme~\ref{Theo:Class}); l'inexistence d'une autre fibration suit alors du lemme~\ref{Lem:PasDautrefibration}.
\qed\end{proof}
\section{D\'emonstrations des th\'eor\`emes}\label{Sec:DemoThms}
\setcounter{subsection}{1}
\begin{proof}[D\'emonstration du th\'eor\`eme~\ref{Theo:Class}]
La proposition~\ref{Prp:Contenu} montre que tout sous-groupe alg\'ebrique du groupe de Cremona est contenu dans un groupe de l'une des familles du th\'eor\`eme~\ref{Theo:Class}. Ceci implique alors que tout groupe alg\'ebrique maximal est \'egal \`a l'un de ces groupes. 

La maximalit\'e de chacun de groupes et les deux derni\`eres assertions du th\'eor\`eme suivent alors de la proposition~\ref{Prp:Superrigidite}. 
\qed\end{proof}

\begin{proof}[D\'emonstration du th\'eor\`eme~\ref{Thm:Param}]
Le th\'eor\`eme~\ref{Theo:Class} \'etant \'etabli, il s'agit maintenant de montrer que les descriptions des groupes d'automorphismes et les param\'etrisations du th\'eor\`eme~\ref{Thm:Param} sont correctes.

Les groupes d'automorphismes des surfaces de del Pezzo proviennent des descriptions faites \`a la section~\ref{Sec:AutdP} et ceux des fibr\'es en coniques de la section~\ref{Sec:GaFibr}.

En ce qui concerne les param\'etrisations des surfaces de del Pezzo, les seuls cas qui ne d\'ecoulent pas directement du th\'eor\`eme~\ref{Theo:Class} sont les surfaces de degr\'e~$2$ et $3$. Dans le premier cas, $S$ est le rev\^etement double de $\Pn$ le long d'une quartique lisse (\ref{SubSec:dP2}) et il faut alors observer qu'une courbe exceptionnelle de $S$ correspond \`a une bitangente de la quartique. Pour les surfaces cubiques, la proposition~\ref{Prp:CubicMin} et le lemme~\ref{Lem:Min3pts} donnent les trois familles de surfaces qui ont des groupes d'automorphismes qui sont des sous-groupes alg\'ebriques maximaux du groupe de Cremona. La premi\`ere famille est de la forme $W^3=L_3(X,Y,Z)$, o\`u $L_3$ est l'\'equation d'une cubique lisse de $\Pn$, que l'on peut mettre sous forme de Hesse $X^3+Y^3+Z^3+\lambda XYZ$. Deux cubiques lisses du plan \'etant isomorphes si et seulement si il existe un automorphisme du plan qui envoie l'une sur l'autre, la classe d'isomorphisme de la courbe cubique d\'etermine celle de la surface cubique. La deuxi\`eme famille ne comporte qu'un \'el\'ement (voir la proposition~\ref{Prp:CubicMin}) et la derni\`ere est celle des surfaces isomorphes \`a 
$W^3+W(X^2+Y^2+Z^2)+\lambda XYZ=0$ avec $9 \lambda^3\not=8\beta$, $8 \lambda^3\not=-1$. La classe d'isomorphisme au sein de cette famille est d\'etermin\'ee par le param\`etre $\lambda$, \`a un facteur $-1$ pr\`es. Ceci se d\'etermine en cherchant les automorphismes de $\mathbb{P}^3$ qui envoient une surfaces cubique sur une autre; ceux-ci devant normaliser le sous-groupe diagonal de $\Aut(S)$ isomorphe \`a $(\Z{2})^2$, ils sont \'egaux \`a une composition de permutations des variables et de matrices diagonales. Les automorphismes obtenus sont alors des compos\'ees d'automorphismes de $S$ avec le changement de variable $X\mapsto -X$.

La param\'etrisation des classes d'isomorphisme de surfaces de Hirzebruch $\mathbb{F}_n$ est donn\'ee par l'entier $n$.

Celle des fibr\'es en coniques exceptionnels est donn\'ee par les points de $\mathbb{P}^1$ ayant des fibres singuli\`eres, modulo l'action de $\Aut(\mathbb{P}^1)$ (Corollaire~\ref{Coro:IsoFibCExcp}). 

Finalement, celle des $(\Z{2})^2$-fibr\'es exceptionnels est donn\'ee par le triplet de points de ramification (proposition~\ref{Prp:Z2FibRam}). 
\qed\end{proof}
\section{$(\Z{2})^2$-fibr\'es en coniques donnant des sous-groupes alg\'ebriques maximaux}\label{Sec:Z2Max}
\subsection{Fibr\'es en coniques avec $6$ ou $7$ fibres singuli\`eres sur une surface qui ne soit pas de del Pezzo}\label{Sec:Z2PasDel}
Soit $\pi:S\rightarrow \mathbb{P}^1$ un $(\Z{2})^2$-fibr\'e en coniques avec $k$ fibres singuli\`eres. Le groupe $\Aut(S,\pi)$ est un sous-groupe alg\'ebrique maximal du groupe de Cremona, si et seulement si la surface $S$ n'est pas de del Pezzo (th\'eor\`eme~\ref{Theo:Class}). Cette derni\`ere condition implique que $k\geq 6$ et est toujours v\'erifi\'ee si $k\leq 8$ (proposition~\ref{Prp:estdPezzofibres}).

Nous d\'ecrivons dans cette section les cas possibles de $(\Z{2})^2$-fibr\'es en coniques avec $6,7$ fibres singuli\`eres. Ceci r\'epond notamment \`a une question de \cite[Section "What is left?"]{bib:DoI} sur l'existence de tels groupes d'automorphismes.

La proposition suivante d\'etermine la g\'eom\'etrie des $(\Z{2})^2$-fibr\'e en coniques ayant $6$ ou $7$ fibres singuli\`eres qui donnent lieu \`a des sous-groupes alg\'ebriques maximaux du groupe de Cremona. Nous donnons ensuite un moyen de construire ceux-ci

\begin{prop}\label{Z267fibresGeometrie}
Soit $\pi:S\rightarrow \mathbb{P}^1$ un $(\Z{2})^2$-fibr\'e en coniques avec $k$ fibres singuli\`eres, $k\in\{6,7\}$, et un triplet de ramification avec $2a_1,2a_2,2a_3$ points (avec $k=a_1+a_2+a_3$). Alors, le groupe $\Aut(S,\pi)$ est un sous-groupe alg\'ebrique maximal du groupe de Cremona si et seulement si on est dans l'une des situations suivantes:
\begin{itemize}
\item[$(a)$]
l'un des trois nombres $a_i$ est \'egal \`a $1$, i.e. $\{a_1,a_2,a_3\}$ est \'egal \`a $\{1,1,4\}$, $\{1,2,3\}$, $\{1,1,5\}$, $\{1,2,4\}$ ou $\{1,3,3\}$.
\item[$(b)$]
il existe $4$ sections $s_0,...,s_3$  de $\pi$, d'auto-intersection $-2$, permut\'ees transitivement par $\Aut(S/\mathbb{P}^1)$ et un morphisme birationnel $\eta:S\rightarrow \Pn$ qui est l'\'eclatement des points $q,p_1,...,p_k$ du plan, qui envoie une fibre g\'en\'erale de $\pi$ sur une droite passant par $q$, qui envoie l'ensemble des quatres sections $s_0,..,s_3$ sur un ensemble de quatre courbes \`a croisements simples transverses, et qui satisfait l'une des assertions suivantes:
\begin{enumerate}
\item[$(b1)$]
le morphisme $\eta$ envoie les sections $s_0,...,s_3$ sur $4$ droites ne passant pas par $q$. De plus, $k=6$, $\{a_1,a_2,a_3\}=\{2,2,2\}$ et les points $p_1,...,p_6$ sont les $6$ points du plan appartenant \`a deux des $4$ droites.
\item[$(b2)$]
le morphisme $\eta$ envoie les sections $s_0,...,s_3$ sur trois droites ne passant pas par $q$ et une conique passant par $q$. 
De plus, $k=7$, $\{a_1,a_2,a_3\}=\{2,2,3\}$, et les points $p_1,...,p_7$ sont $7$ des $9$ points du plan appartenant \`a deux des $4$ courbes. Les $2$ points restants sont align\'es avec $q$.
\end{enumerate}
\end{itemize}
\end{prop}
\begin{proof}
D'apr\`es le th\'eor\`eme~\ref{Theo:Class}, le groupe $\Aut(S,\pi)$ est un sous-groupe alg\'ebrique maximal du groupe de Cremona si et seulement si $S$ n'est pas une surface de del Pezzo. La proposition~\ref{Prp:estdPezzofibres} montre que ceci est vrai si et seulement si une des situations suivantes se pr\'esente:
$(i)$ une courbe fix\'ee par une involution de $\Aut(S,\pi)$ est rationnelle;
$(ii)$ $S$ contient $4$ sections d'auto-intersection~$-2$. La condition $(i)$ est \'equivalente \`a ce que $\{a_1,a_2,a_3\}$ contiennent un $1$. Comme $a_1+a_2+a_3=k$ (proposition~\ref{Prp:Z2FibRam}), on retrouve les $5$ possibilit\'es d\'ecrites en $(a)$.

Supposons maintenant que $S$ contient une section $s_0$ d'auto-intersection $-2$ et que $a_1=a_2=2$ (ce qui implique que $a_3=2$ si $k=6$ ou que $a_3=3$ si $k=7$) et d\'emontrons l'existence d'un morphisme birationnel $\eta\rightarrow \Pn$ satisfaisant les conditions de $(b)$, ce qui terminera la preuve.

Notons $\sigma_1,\sigma_2,\sigma_3$ les trois involutions de $\Aut(S/\mathbb{P}^1)$, qui permutent les composantes de respectivement $a_1,a_2$ et $a_3$ fibres singuli\`eres et notons $s_i=\sigma_i(s_0)$ pour $i=1,..,3$. D\'ecomposons l'ensemble des $k$ fibres singuli\`eres de $\pi$ en trois parties disjointes $f_{1}$, $f_{2}$, $f_{3}$, telles que chacune des composantes de chaque fibre de $f_i$ est laiss\'ee invariante par $\sigma_i$. En particulier, si $k=6$ chacune des parties $f_i$ contient deux fibres et si $k=7$, alors $f_1,f_2,f_3$ contiennent respectivement $1,3,3$ fibres.

On construit un morphisme birationnel de fibr\'es en coniques $\nu:S\rightarrow \mathbb{F}_n$ qui contracte une composante dans chaque fibre singuli\`ere de $\pi$. Pour cela, on choisit dans chaque $f_i$ une des fibres o\`u l'on contracte la composante qui touche $s_0$ (et donc aussi $s_i$) et dans les autres fibres de $f_i$, on contracte la composante qui ne touche pas $s_0$ (et donc ne touche pas $s_i$). En calculant combien de courbes contract\'ees touchent chacune des sections, on voit que $s_0,s_1,s_2$ sont envoy\'ees sur des sections de $\mathbb{F}_n$ d'auto-intersection $1$ et que $s_3$ est envoy\'ee sur une section d'auto-intersection $1$ si $k=6$ et $3$ si $k=7$. L'existence de sections d'auto-intersection $1$ implique que $n=1$;
la contraction de la section d'auto-intersection~$-1$ sur un point $q\in\Pn$ induit un morphisme birationnel $\eta:S\rightarrow \Pn$ qui envoie une fibre g\'en\'erale de $\pi$ sur une droite passant par $q$. Les sections d'auto-intersection $1$ (respectivement $3$) de $\mathbb{F}_1$ sont envoy\'ees sur des droites ne passant pas par $q$ (respectivement sur des coniques passant par $q$). De plus, les $k$ courbes contract\'ees par $\nu$ touchant chacune exactement $2$ sections, ces $k$ courbes sont envoy\'ees sur des points \`a l'intersection d'exactement $2$ des courbes images. Ceci implique que trois de ces courbes ne s'intersectent pas en un m\^eme point. Lorsque $k=6$, les $4$ sections sont envoy\'ees sur $4$ droites ne passant pas par $q$, celles-ci s'intersectent en au plus $6$ points, il n'y pas d'intersection triple et on obtient la situation $(b1)$. Lorque $k=7$, les courbes images sont $3$ droites et une conique, qui ont au plus $9$ points d'intersection, $7$ \'etant des intersections transverses (les point $p_1,..,p_7$). Le(s) point(s) d'intersection restant deviennent sur $S$ des points de m\^eme type d'intersection o\`u $\Aut(S/\mathbb{P}^1)$ agit. Les $4$ sections \'etant disjointes et permut\'ees transitivement par $\Aut(S/\mathbb{P}^1)\cong (\Z{2})^2$, il doit y avoir $2$ points, permut\'es par $\Aut(S/\mathbb{P}^1)$, qui donnent sur $\Pn$ deux points align\'es avec $q$.
\qed\end{proof}

Rappelons le r\'esultat classique suivant, que nous utiliserons pour montrer l'existence de $(\Z{2})^2$-fibr\'es en coniques correspondant aux situations $(b1)$ ou $(b2)$ de la proposition~\ref{Z267fibresGeometrie}.
\begin{lemm}\label{Lem:invoDroites}
Soit $\eta:S\rightarrow \Pn$ l'\'eclatement des points $q,p_1,p_2,p_3,p_4\in \Pn$, tels que $S$ soit de del Pezzo et soit $\pi:S\rightarrow \mathbb{P}^1$ la fibration en coniques induite par la projection de $\Pn$ depuis $q$.

Alors, il existe une involution $\sigma\in \Aut(S,\pi)$, qui agit trivalement sur la base de la fibration et qui \'echange les composantes de chacune des $4$ fibres singuli\`eres.
En notant $E_q=\eta^{-1}(q)$, $E_1=\eta^{-1}(p_1)$,..., $E_4=\eta^{-1}(p_4)$ et $L$ le transform\'e d'une droite g\'en\'erale de $\Pn$, l'action de $\sigma$ sur $\Pic{S}$ est donn\'ee par
\begin{center}
$\left(\begin{array}{rrrrrr}
-1 & -1 & -1& -1 & -1 & -2\\
-1& -1 & 0& 0 & 0 & -1\\
-1 & 0& -1& 0 & 0 & -1\\
-1 & 0& 0& -1 & 0 & -1\\
-1 & 0& 0& 0& -1 & -1\\
2 & 1& 1& 1 & 1 & 3\end{array}\right)$\end{center} relativement \`a la base $(E_q,E_1,E_2,E_3,E_4,L)$.
Le transform\'e strict sur $S$ de la droite passant par deux points de $\{p_1,p_2,p_3,p_4\}$ est envoy\'e par $\sigma$ sur le transform\'e strict de la droite passant par les deux autres points.
\end{lemm}
\begin{proof}Notons $C$ la cubique de $\Pn$ passant par $q$, et par chaque point $p_i$, de mani\`ere tangente \`a la droite passant par $q$ et $p_i$. Notons $\sigma'$ l'involution birationnelle de $\Pn$ qui fixe chaque point de $C$ et pr\'eserve le pinceau de droites par~$q$. La restriction de $\sigma'$ \`a une droite g\'en\'erale passant par $q$ est une involution qui fixe les deux autres points de la droite qui appartiennent \`a $C$. Cette construction classique d'involutions de de Jonqui\`eres se trouve par exemple dans \cite[Exemple 2.4c]{bib:BaB}. Alors, $\sigma'$ a exactement $5$ points-bases, qui sont $q,p_1,...,p_4$ et donc  $\sigma=\eta^{-1}\sigma'\eta\in \Aut(S,\pi)$. 

Comme $C$ est une courbe elliptique et une bisection de $\pi$, elle est ramifi\'ee en $4$ points, ce qui implique que $\sigma$ permute les deux composantes de chacune des $4$ fibres singuli\`ere de $\pi$, donc $\sigma(E_i)=L-E_q-E_i$ pour $i=1,..,4$. Comme $\sigma$ est un automorphisme et pr\'eserve la fibration on a $\sigma(K_S)=K_S=-3L+E_q+E_1+...+E_4$ et $\sigma(L-E_q)=L-E_q$, ce qui implique la matrice de l'\'enonc\'e. 

Il existe une unique courbe irr\'eductible de $S$ \'equivalente \`a $L-E_1-E_2$, qui est le transform\'e strict de la droite passant par $p_1$ et $p_2$, la situation est similaire en prenant deux autres points. La derni\`ere assertion suit donc du calcul direct $\sigma(L-E_1-E_2)=L-E_3-E_4$.
\qed\end{proof}

D\'emontrons maintenant l'existence de $(\Z{2})^2$-fibr\'es en coniques correspondant \`a la situation $(b1)$ de la proposition~\ref{Z267fibresGeometrie}.
\begin{prop}\label{prop:Fibre4droites}
Soit $Q\subset \Pn$ la r\'eunion de $4$ droites distinctes, telles que trois ne s'intersectent pas en un m\^eme point, et soit $q\in\Pn\backslash Q$.
Notons $\eta:S\rightarrow \Pn$ l'\'eclatement de $q$ et des $6$ points  singuliers de $Q$.

Alors, la projection de $\Pn$ depuis $q$ se remonte \`a une fibration en coniques $\pi:S\rightarrow \mathbb{P}^1$, qui est un $(\Z{2})^2$-fibr\'e en coniques.  Le transform\'e strict de $Q$ sur $S$ donne $4$ sections disjointes, d'auto-intersection~$-2$. Le groupe $G=\Aut(S,\pi)$ est un sous-groupe alg\'ebrique maximal du groupe de Cremona.
\end{prop}
\begin{proof}
Notons $\Delta \subset Q$ l'ensemble des $6$ points singuliers de $Q$. Cet ensemble contient trois paires: chaque point appartient \`a $2$ droites, on lui associe le point d'intersection des deux autres droites. 

L'\'eclatement de deux paires de $\Delta$ et du point $q$ donne une surface de del Pezzo de degr\'e~$4$. Sur celle-ci, il existe un automorphisme d'ordre $2$ qui agit trivialement sur la base de la fibration en coniques induite par la projection de de $\Pn$ depuis $q$ et permute les composantes des $4$ fibres singuli\`eres (voir le lemme~\ref{Lem:invoDroites}). Cette involution agit sur le transform\'e strict de $Q$ (lemme~\ref{Lem:invoDroites}) et donc fixe les deux points de la troisi\`eme paire. L'involution se remonte alors en un \'el\'ement de $\Aut(S,\pi)$ qui permute les composantes de $4$ fibres singuli\`eres.

En faisant ceci pour les trois choix possibles de deux paires, on en d\'eduit que $(S,\pi)$ est un $(\Z{2})^2$-fibr\'e exceptionnel. Chaque droite de $Q$ passant par $2$ points \'eclat\'es par $\eta$, le transform\'e strict de $Q$ sur $S$ donne $4$ sections disjointes, d'auto-intersection~$-2$. La surface n'est donc pas de del Pezzo et alors $\Aut(S,\pi)$ est un sous-groupe alg\'ebrique maximal du groupe de Cremona (th\'eor\`eme~\ref{Theo:Class}).
\qed\end{proof}

Nous faisons maintenant de m\^eme pour la situation $(b2)$ de la proposition~\ref{Z267fibresGeometrie}.

\begin{prop}\label{prop:Fibre3droitesUneConique}
Soit $Q\subset \Pn$ un ensemble de trois droites et une conique, tel que trois des quatre courbes ne s'intersectent pas en un m\^eme point et que la conique intersecte chaque droite en deux points distincts.
Parmi les $9$ points singuliers de $Q$ on en choisit deux: $d_1,d_2$ tels que $d_1$ soit sur deux droites et $d_2$ soit sur la $3$-\`eme droite et sur la conique. On note $q$ le point de la conique qui est align\'e avec $d_1$ et $d_2$ et $\eta:S\rightarrow \Pn$ l'\'eclatement de $q$ et des $7$ points  singuliers de $Q$ qui ne sont ni $d_1$ ni $d_2$ (voir la figure \ref{FigCLaLbLc}). 
Supposons que $q$ n'est pas align\'e avec deux points singuliers de $Q$ (hormis $d_1$ et $d_2$). 

Alors, la projection de $\Pn$ depuis $q$ se remonte \`a une fibration en coniques $\pi:S\rightarrow \mathbb{P}^1$, qui est un $(\Z{2})^2$-fibr\'e en coniques. Le transform\'e strict de $Q$ sur $S$ donne $4$ sections disjointes, d'auto-intersection~$-2$. Le groupe $G=\Aut(S,\pi)$ est un sous-groupe alg\'ebrique maximal du groupe de Cremona.
\end{prop}
\begin{proof}
Notons $L_a$ et $L_b$ les deux droites de $Q$ qui touchent $d_2$, $L_c$ la troisi\`eme droite et $C$ la conique. Les $7$ points singuliers de $Q$ \'eclat\'es par $\eta$ sont not\'es $a_1,a_2,a_3,b_1,b_2,b_3,c$ de la mani\`ere suivante: $a_1,a_2,a_3$ sont sur $L_a$; $b_1,b_2,b_3$ sont sur $L_b$; $a_3,b_3,c$ sont sur $L_c$ et $a_1,a_2,b_1,b_2,c$ sont sur $C$ (voir la figure \ref{FigCLaLbLc}).

\begin{figure}{\begin{center}
\includegraphics[width=7cm]{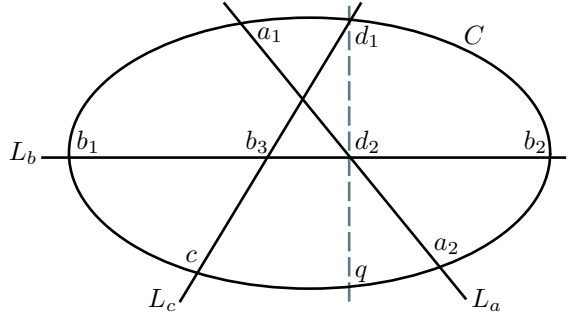}
\drawat{-29.5mm}{34.5mm}{$d_1$}%
\drawat{-42.5mm}{35.2mm}{$a_1$}%
\drawat{-19mm}{7mm}{$a_2$}%
\drawat{-14mm}{-0.5mm}{$L_a$}%
\drawat{-57mm}{-0.5mm}{$L_c$}%
\drawat{-75.5mm}{19mm}{$L_b$}%
\drawat{-15mm}{34.5mm}{$C$}%
\drawat{-52mm}{5.5mm}{$c$}%
\drawat{-44mm}{20.5mm}{$b_3$}%
\drawat{-7mm}{20.5mm}{$b_2$}%
\drawat{-66.5mm}{20.5mm}{$b_1$}%
\drawat{-29.5mm}{3.5mm}{$q$}%
\drawat{-29.5mm}{20.5mm}{$d_2$}%
\end{center}} \caption{\small La configuration des points sur $Q$.\label{FigCLaLbLc}}\end{figure}

L'\'eclatement des points $q$, $a_1$, $a_2$, $b_3$ et $c$ donne une surface de del Pezzo de degr\'e~$4$, o\`u il existe un automorphisme $\sigma_a$ d'ordre $2$, qui agit trivialement sur la base de la fibration en coniques induite par la projection de de $\Pn$ depuis $q$ et permute les composantes des $4$ fibres singuli\`eres (voir le lemme~\ref{Lem:invoDroites}). Alors, $\sigma_a$ permute les transform\'es stricts de $L_a$ et $L_c$ (lemme~\ref{Lem:invoDroites}) et permute donc les points $d_1$ et $d_2$ et fixe $b_3$. La transform\'ee stricte de $L_b$ est envoy\'ee sur la transform\'ee stricte d'une conique $D$ passant par $q$, $a_1$, $a_2$, $c$ (voir la matrice du lemme~\ref{Lem:invoDroites}); comme $L_b$ passe par $d_2$, la conique $D$ passe par $d_1$ et est donc \'egale \`a $C$. Les transform\'es stricts de $L_b$ et $C$ sont donc permut\'es par $\sigma_a$, qui fixe alors $b_1$ et $b_2$. L'involution $\sigma_a$ se remonte alors \`a un \'el\'ement d'ordre $2$ de $\Aut(S,\pi)$ qui agit trivialement sur la base de la fibration, permute les composantes de $4$ fibres singuli\`eres et agit sur le transform\'e strict de $Q$.

En faisant de m\^eme avec les points $q$, $b_1$, $b_2$, $a_3$ et $c$ (i.e. en \'echangeant le r\^ole de $L_a$ et $L_b$), on obtient une deuxi\`eme involution. Les deux engendrent un groupe isomorphe \`a $(\Z{2})^2$. On en d\'eduit que $(S,\pi)$ est un $(\Z{2})^2$-fibr\'e exceptionnel. Le transform\'e strict de $Q$ sur $S$ donne $4$ sections d'auto-intersection~$-2$. La surface n'est donc pas de del Pezzo et alors $\Aut(S,\pi)$ est un sous-groupe alg\'ebrique maximal du groupe de Cremona (th\'eor\`eme~\ref{Theo:Class}).
\qed\end{proof}

\subsection{Param\'etrisation par des vari\'et\'es alg\'ebriques}
La proposition \ref{Prp:Z2FibRam} implique les r\'esultats suivants.
Pour tout $n\geq 3$, l'ensemble des classes d'isomorphismes de $(\Z{2})^2$-fibr\'es en coniques avec $n$ fibres singuli\`eres est param\`etr\'e par une vari\'et\'e alg\'ebrique $\mathcal{C}_n$ de dimension $n-3$, correspondant aux triplets de points de ramification \`a action de $\Aut(\mathbb{P}^1)$ pr\`es; l'ensemble $\mathcal{C}_n$ a plusieurs composantes connexes (toutes de dimension $n-3$), correspondant aux d\'ecompositions possible de $n$ en trois entiers strictement positifs $a_1,a_2,a_3$ (pour un triplet $\{A_1,A_2,A_3\}$, le nombre de points de $A_i$ est $2a_i$ et $a_1+a_2+a_3=n$), on note $\mathcal{C}_{a_1,a_2,a_3}$ les composantes, avec $a_1\leq a_2\leq a_3$; par exemple $\mathcal{C}_6=\mathcal{C}_{1,1,4}\cup\mathcal{C}_{1,2,3}\cup\mathcal{C}_{2,2,2}$ et $\mathcal{C}_7=\mathcal{C}_{1,1,5}\cup\mathcal{C}_{1,2,4}\cup\mathcal{C}_{1,3,3}\cup\mathcal{C}_{2,2,3}$.

On note $\mathcal{C\! M}_n\subset \mathcal{C}_n$ le sous-ensemble donnant des sous-groupes alg\'ebriques maximaux du groupe de Cremona. Cet ensemble param\`etre les classes de conjugaisons de tels sous-groupes et correspond aux fibr\'es en coniques dont la surface ambiante n'est pas de del Pezzo (th\'eor\`eme \ref{Thm:Param}, famille $(11)$). Les propositions \ref{Z267fibresGeometrie}, \ref{prop:Fibre4droites} et \ref{prop:Fibre3droitesUneConique} permettent de d\'eterminer totalement $\mathcal{C\! M}_n$, nous r\'esumons ceci dans la proposition suivante.
\begin{prop}
Pour $n\leq 5$, $\mathcal{C\! M}_n$ est vide. Pour $n\geq 6$, l'ensemble $\mathcal{C\! M}_n$ est une vari\'et\'e alg\'ebrique. De plus:
\begin{enumerate}
\item
$\mathcal{C\! M}_6$ est la r\'eunion de  $\mathcal{C}_{1,1,4}$, $\mathcal{C}_{1,2,3}$ et de l'hypersurface de $\mathcal{C}_{2,2,2}$  correspondant aux fibr\'es d\'ecrits \`a la proposition \ref{prop:Fibre4droites}.
\item
$\mathcal{C\! M}_7$ est la r\'eunion de  $\mathcal{C}_{1,1,5}$, $\mathcal{C}_{1,2,4}$, $\mathcal{C}_{1,3,3}$ et de l'hypersurface de $\mathcal{C}_{2,2,3}$ correspondant aux fibr\'es d\'ecrits \`a la proposition \ref{prop:Fibre3droitesUneConique}.
\item
pour tout $n\geq 8$, $\mathcal{C\! M}_n=\mathcal{C}_n$.
\end{enumerate} 
\end{prop}
\begin{proof}
La description de $\mathcal{C\! M}_n$ pour $n\not=6,7$ est donn\'ee par la proposition \ref{Prp:estdPezzofibres}; pour $n=6,7$, elle suit des propositions \ref{Z267fibresGeometrie}, \ref{prop:Fibre4droites} et \ref{prop:Fibre3droitesUneConique}.
\qed\end{proof}
\section{$(\Z{2})^2$-fibr\'es en coniques tels que le groupe d'automorphismes de la surface ambiante ne soit pas alg\'ebrique}\label{Sec:Z2pasAlg}

\begin{prop}
Soit $(S,\pi)$ un $(\Z{2})^2$-fibr\'e en coniques, avec un tripl\'e de ramification $\{A_1,A_2,A_3\}$, tel que $A_1$ et $A_2$ comptent $4$ points et $A_3=A_1\cup A_2$ compte $8$ points.

Alors, deux involutions de $\Aut(S/\mathbb{P}^1)$ fixent chacune une courbe elliptique. Les deux courbes elliptiques ne se touchent pas et sont toutes deux \'equivalentes \`a $-K_S$; le morphisme anti-canonique induit une fibration elliptique $\eta:S\rightarrow \mathbb{P}^1$, invariante par $\Aut(S)$. 
 
 Le groupe $\Aut(S)$ n'est pas un groupe alg\'ebrique, alors que le groupe $\Aut(S,\pi)$ -- sous-groupe de $\Aut(S)$ -- est un sous-groupe alg\'ebrique maximal du groupe de Cremona. 
\end{prop}
\begin{proof}
Notons $\sigma_1,\sigma_2,\sigma_3$ les trois involutions de $\Aut(S/\mathbb{P}^1)$ fixant respectivement les courbes $C_1,C_2,C_3$, ramifi\'ees aux points $A_1,A_2,A_3$. Alors, $C_1$, $C_2$ et $-K_S$ sont \'egaux dans  $\Pic{S}$ et leur carr\'e est \'egal \`a $0$ (proposition~\ref{Prp:Z2FibRam}). Les courbes $C_1$ et $C_2$ \'etant diff\'erentes, le syst\`eme lin\'eaire des courbes \'equivalentes \`a $-K_S$ est de dimension projective au minimum $1$. Comme $(K_S)^2=0$, la dimension est exactement $1$ et le morphisme anti-canonique donne une fibration elliptique $\eta:S\rightarrow \Pn$, qui est invariante par $\Aut(S)$ (de mani\`ere classique cette surface est une surface de Halphen d'indice $1$).

La surface $S$ n'\'etant pas de del Pezzo, le groupe $\Aut(S,\pi)$ est un sous-groupe alg\'ebrique maximal du groupe de Cremona (th\'eor\`eme~\ref{Theo:Class}). Il reste \`a montrer que $\Aut(S)$ n'est pas alg\'ebrique. Pour cela, on choisit deux fibres singuli\`eres de $\pi$ et l'on prend une composante dans chaque fibre, ce qui donne deux sections disjointes de $\eta$. Alors, la translation de l'une des deux sections sur l'autre est un automorphismes de $S$, qui est d'ordre infini et qui agit \'egalement de mani\`ere infinie sur $\Pic{S}$, ce qui termine la preuve.
\qed\end{proof}

            \end{document}